\documentclass[11pt]{article}
\usepackage[disable]{todonotes}
\usepackage{amsmath}
\usepackage{amsthm}
\usepackage{amssymb}
\usepackage{mathrsfs}
\usepackage{mathtools}
\usepackage{cite}
\usepackage{amscd}
\usepackage{authblk}
\usepackage{eufrak}
\usepackage{fancyhdr}
\usepackage[colorlinks=true]{hyperref}
\usepackage{xcolor}
\usepackage[all]{xy}

\usepackage[utf8]{inputenc}
\usepackage[T1]{fontenc}

\theoremstyle{plain}
\newtheorem{thm}{Theorem}[section]

\newtheorem{lem}[thm]{Lemma}
\newtheorem{cor}[thm]{Corollary}

\newtheorem{conj}[thm]{Conjecture}

\theoremstyle{definition}
\newtheorem{definition}[thm]{Definition}
\newtheorem{rem}[thm]{Remark}
\newtheorem{ex}[thm]{Example}
\newtheorem*{ack}{Acknowledgement}
\numberwithin{equation}{section}
\usepackage{amssymb}
\usepackage[top=25truemm, bottom=25truemm, left=25truemm, right=25truemm]{geometry}
\newcommand{\Z}{\mathbb{Z}}
\newcommand{\Q}{\mathbb{Q}}
\newcommand{\N}{\mathbb{N}}
\newcommand{\Spec}{\operatorname{Spec}}

\newcommand{\Hom}{{\rm Hom}}

\newcommand{\Dim}{\operatorname{dim}}
\newcommand{\End}{{\rm End}}
\newcommand{\GL}{\mathrm{GL}}
\newcommand{\SL}{\mathrm{SL}}

\newcommand{\fgl}{\mathfrak{gl}}
\newcommand{\fsl}{\mathfrak{sl}}

\newcommand{\fg}{\mathfrak{g}}

\newcommand{\Lie}{\operatorname{Lie}}

\newcommand{\Ker}{{\rm Ker}}
\newcommand{\im}{{\rm Im}}

\newcommand{\Det}{{\rm det}}

\newcommand{\Def}{\overset{{\rm def}}{=}}
\newcommand{\et}{\text{{\rm \'et}}}
\newcommand{\loc}{{\rm loc}}
\newcommand{\Nori}{{\rm N}}

\newcommand{\fppf}{{\rm fppf}}

\newcommand{\G}{\mathbb{G}}
\newcommand{\F}{\mathbb{F}}
\newcommand{\A}{\mathbb{A}}
\newcommand{\C}{\mathbb{C}}
\newcommand{\calM}{\mathcal{M}}
\newcommand{\bfp}{\mathbf{p}}
\newcommand{\bfn}{\mathbf{n}}
\newcommand{\unit}{\mathbf{1}}
\newcommand{\Der}{\mathrm{Der}}
\newcommand{\lto}{\longrightarrow}
\newcommand{\ad}{\mathrm{ad}}

\newcommand{\gr}{\operatorname{gr}}

\definecolor{linkblue}{HTML}{0000EE}

\title{\bf A generalized Abhyankar's conjecture for\\
simple Lie algebras in characteristic $p>5$}
\author{Shusuke Otabe, Fabio Tonini and Lei Zhang}
\date{
\vspace{-5mm}
}

\begin{document}

\maketitle

\begin{abstract}
In the present paper, we study a purely inseparable counterpart of Abhyankar's
conjecture for the affine line in positive characteristic, and prove its validity for all the finite local non-abelian simple group schemes in characteristic $p>5$. The crucial point is how to deal with finite local group schemes which cannot be realized as the Frobenius kernel of a smooth algebraic group. Such group schemes  appear as the ones associated with Cartan type Lie algebras.  We settle the problem for such Lie algebras by making use of natural gradations or filtrations on them. \end{abstract}

\thispagestyle{fancy}
\renewcommand{\headrulewidth}{0pt}
\renewcommand{\footrulewidth}{0.1mm}

\lfoot{{\footnotesize \textit{Date}: \today. 2nd version.
~
}}



{
\hypersetup{linkcolor=linkblue}
\tableofcontents

}


\section{Introduction}

Let $k$ be an algebraically closed field of characteristic $p>0$ and $U$ an affine smooth connected curve defined over $k$. 
In \cite{ab57}, Abhyankar proposed a conjectural answer to the \textit{inverse
Galois problem} for finite \'etale covers of $U$, \emph{i.e.}\ the question of which
finite group $G$ occurs as the Galois group of a finite \'etale connected
Galois cover over $U$. In terms of Grothendieck's \textit{\'etale fundamental
groups}~\cite{gr71}, Abhyankar's conjecture can be reformulated as a
description of the set
\[
    \pi_A^\et(U) = \{\textup{finite groups which appear as quotients
    of $\pi_1^{\et}(U)$}\}.
\]
This was affirmatively proved by Raynaud for the affine line $U=\A^1_k$~(cf.\
\cite{ra94}), and by Harbater for general $U$~(cf.\ \cite{ha94}). In the particular case when $U=\A^1_k$, the conjecture can be stated in the following way.  

\begin{thm}{\rm (Raynaud, cf.\ \cite{ra94})}\label{thm:raynaud}
The $\pi_A^\et(\A^1_k)$ is the set of \textit{quasi-$p$}-groups, \emph{i.e.}\ groups generated by all their $p$-Sylow subgroups. 
\end{thm}

The following are several examples.  

\begin{ex}~\label{ex:quasi-p}
\begin{enumerate}
\renewcommand{\labelenumi}{(\arabic{enumi})}
\item Any $p$-group is a quasi-$p$-group. For $p$-groups, Theorem \ref{thm:raynaud} is an immediate consequence of the fact that the maximal pro-$p$-quotient $\pi^{\et}_1(\A^1_k)^{(p)}$ is isomorphic to the free pro-$p$-group of rank $\# k\ge \#\N$. 

\item For any integer $n\ge 3$, the $n$-th symmetric group $\mathfrak{S}_n$ is
    not a $2$-group, but a quasi-$2$-group. If $n\ge 5$, the $n$-th alternating
    group $\mathfrak{A}_n$ is a  non-abelian simple group. Thus, for any prime
    number $p$ with  $p\mid \#\mathfrak{A}_n$, the simple group
    $\mathfrak{A}_n$ is a quasi-$p$-group. For any positive integer $t$ with
    $p\nmid t$, Abhyankar constructed a finite \'etale connected Galois cover
    over $\A^1_k$ whose Galois group is isomorphic to $\mathfrak{S}_{2+t}$ if
    $p=2$, and to $\mathfrak{A}_{p+t}$  if $p$ and  $t$ are $\ge 3$. See \cite[Theorem 3.4]{ha18} for the detail. 

\item Let $\Sigma$ be a semisimple simply connected algebraic group defined over a finite field $\F_q$ of characteristic $p$. Then the finite group $\Sigma(\F_q)$ is a quasi-$p$-group which is not a $p$-group. In \cite{no94}, Nori proved that $\Sigma(\F_q) \in \pi_A^\et(\A^1_k)$ (cf.\ \cite{no94}\cite[\S3.2]{se92}\cite[Theorem 3.5]{ha18}). 
\end{enumerate}
\end{ex}

Raynaud's proof of Theorem \ref{thm:raynaud} is done by the induction on the order $\# G$ of a given quasi-$p$-group $G$. If $G$ contains a nontrivial normal $p$-subgroup $1\neq N\triangleleft~G$, then by induction hypothesis, one can take a surjective homomorphism $\pi^{\et}_1(\A^1_k)\twoheadrightarrow G/N$, in which case, by applying Serre's solution~\cite{se90} to the \textit{embedding problem} 
\begin{equation*}
\begin{xy}
\xymatrix{
&&&\pi^{\et}_1(\A^1_k)\ar@{->>}[d]&\\
1\ar[r]&N\ar[r]&G\ar[r]&G/N\ar[r]&1,
}
\end{xy}
\end{equation*}   
one can conclude that $G\in \pi_A^\et(\A^1_k)$. Thus, one may assume that $G$ has no nontrivial normal $p$-subgroup. For such a quasi-$p$-group, Raynaud adopted the \textit{rigid analytic patching} method and the use of \textit{semi-stable curves} to construct a finite \'etale connected Galois $G$-cover of $\A^1_k$.

In \cite{ot18}, a \textit{purely inseparable analogue} of Abhyankar's conjecture was proposed as a question~(cf.\ \cite[Question 3.3]{ot18}). The analogue is formulated in terms of the Nori's \textit{fundamental group scheme} $\pi^\Nori$~(cf.~\cite{no76}\cite[Chapter II]{no82}). Namely, it gives a hypothetical description of the set 
\[
\pi_A^\loc(U)=\{\text{finite \textit{local} $k$-group schemes which appear as quotients of } \pi^{\Nori}(U)\}
\]
Here by a finite local $k$-group scheme we mean a finite $k$-group scheme which is connected as a scheme (and therefore topologically a point).
In \cite{ot18}, several evidences were also discussed~(cf.\ \cite[Proposition 3.4 and Corollaries 4.15 and 4.19]{ot18}). In the sequel paper \cite{ot19}, the purely inseparable counterpart was affirmatively settled in the \textit{solvable} case~(cf.~\cite[Corollary 4.33]{ot19}). 

The set $\pi_A^\loc(U)$ can be defined for any reduced $k$-scheme $U$ (even without a rational point) as the set of finite $k$-group schemes appearing as quotients of the \emph{local fundamental group scheme} $\pi^\loc(U)$ (cf. Definition \ref{def: local quotients}). This allows to consider also a \emph{generic} Abhyankar's
conjecture: in \cite[Conjecture II]{RTZ} it is conjectured that $\pi^\loc_A(\Spec k(t))$ is the set of all finite local $k$-group schemes. 

In the present paper, we focus on the affine line $U=\A^1_k$, in which case the generalized Abhyankar's conjecture can be stated in the following way.

\begin{conj}{\rm (cf.~\cite[Question 3.3 for the affine line]{ot18})}\label{conj:PIAC A^1}
The $\pi_A^\loc(\A^1_k)$ is the set of finite local $k$-group schemes with no nontrivial characters.  
\end{conj}

Note that the `only if' part is always true~(cf.~\cite[Proposition 3.1]{ot18}).
Conjecture \ref{conj:PIAC A^1} implies that $G\in \pi_A^\loc(\A^1_k)$ if and only if $G^{ab}\in \pi_A^\loc(\A^1_k)$ for a finite local $k$-group scheme $G$, where $G^{ab}$ is the maximal abelian quotient of $G$. We verify Conjecture \ref{conj:PIAC A^1} for finite local non-abelian \textit{simple} $k$-group schemes at least when $k$ has sufficiently large characteristic $p$. Namely, as the main theorem of the present paper, we will prove the following result.

\begin{thm}{\rm (cf.~Corollary \ref{cor:AC Cartan})}\label{thm:main} 
Let $k$ be an algebraically closed field of characteristic $p>5$ and $G$ a finite local non-abelian simple $k$-group scheme. Then $G\in \pi_A^\loc(\A^1_k)$.
\end{thm}

We will also see that Theorem \ref{thm:main} is `almost' true even in characteristic $p=5$~(cf.\ Theorem \ref{thm:AC Cartan}). Moreover, it turns out that `most' of our simple quotients $G$ of $\pi^{\Nori}(\A^1_k)$ admits a form $G_0$ over the prime field $\F_p$, \emph{i.e.}\ $G= G_0\otimes_{\F_p}k$, and that $G_0\in \pi_A^\loc(\A^1_{\F_p})$~(cf.\ Corollaries  \ref{cor:AC chevalley} and \ref{cor:AC Cartan p=2,3}).

Now we explain our strategy for the proof. 
In contrast to Raynaud's geometric approach to Abhyankar's conjecture, our proof of Theorem \ref{thm:main} heavily relies on the classification theorem for finite local simple group schemes. 
If a finite local $k$-group scheme $G$ is simple, then $G$ must be of height one, \emph{i.e.}\ the Frobenius map $F:G\lto G$ is trivial. Recall that the functor $G\longmapsto\Lie(G)$ induces an equivalence of categories
\begin{equation}\label{eq:scheme vs restricted Lie algebra}
\Biggl(
\begin{gathered}
~~\text{finite local $k$-group schemes}~~\\
\text{of height one}
\end{gathered}
\Biggl)~
\overset{~\simeq~}{\longleftrightarrow}
~\Biggl(
\begin{gathered}
~~\text{finite dimensional restricted}~~\\
\text{Lie algebras over $k$}
\end{gathered}
\Biggl)
\end{equation}
(cf.\ \cite[Chapter II, \S 7, n$^\textup{o}$4, 4.1]{dg70}).  
For example, if $G$ is a finite local abelian simple $k$-group scheme, then $G$
is isomorphic to $\alpha_p$ or $\mu_p$. In both the cases, the associated Lie
algebra $\Lie(G)$ is isomorphic to $k$ with the trivial Lie bracket. If $G=\alpha_p$, then the  restricted structure $[p]:k\lto k$ is given by the zero map, \emph{i.e.}\ $x^{[p]}=0$ for any $x\in \Lie(\alpha_p)=k$. If $G=\mu_p$, then the restricted structure $[p]$ is given by the $p$-th power map, \emph{i.e.}\ $x^{[p]}=x^p$ for $x\in k$. 
Via the above equivalence, 
 the classification of finite local non-abelian simple group schemes is equivalent to the
 classification of finite dimensional simple restricted Lie algebras~(cf.\
 \cite[\S2, Corollary 1]{viviani}), \emph{i.e.} there is an equivalence
 of categories
\begin{equation*}
\Biggl\{
\begin{gathered}
~~\text{finite local non-abelian simple}~~\\
\text{$k$-group schemes}
\end{gathered}
\Biggl\}~
\overset{\simeq}{\longleftrightarrow}
~\Biggl\{
\begin{gathered}
\text{finite dimensional simple}\\
~~\text{restricted Lie algebras over $k$}~~
\end{gathered}
\Biggl\}.
\end{equation*}
Moving toward a more precise understanding of the classification, let us recall
the following correspondence~(cf.\ \cite[\S4, Theorem 4]{viviani}) of
\textit{isomorphism classes}.\todo{isomorphism class or equivalence of
categories?}
\begin{equation*}
\Biggl\{
\begin{gathered}
\text{finite dimensional}\\
~~\text{simple restricted Lie algebras over $k$}~~
\end{gathered}
\Biggl\}~
\overset{1:1}{\longleftrightarrow}
~\Biggl\{
\begin{gathered}
~~\text{finite dimensional}~~\\
\text{simple Lie algebras over $k$}
\end{gathered}
\Biggl\},
\end{equation*}
where the map from the left hand side to the right hand side is given by taking the derived subalgebra $L\longmapsto [L,L]$, and the converse is given by taking the \textit{$p$-envelope} $L\longmapsto L_{[p]}$~(cf.~Definition \ref{def:p-env}). Therefore, the classification of finite local non-abelian simple $k$-group schemes is equivalent to the classification of  simple Lie algebras over $k$. The latter has a complete answer at least in characteristic $p>3$, which is known as the \textit{generalized Kostrikin--Shafarevich conjecture} and was established by Block--Wilson--Strade--Premet~(cf.\ \cite{Strade}\cite{viviani}. See also Theorem \ref{thm:KS conj}). For the history of the classification of simple Lie algebras in positive characteristic, see \cite[Introduction]{Strade}.

A basic family of  simple Lie algebras is obtained by taking the reduction modulo $p$ of simple complex Lie algebras. Such a simple Lie algebra is called of \textit{classical type}~(cf.~\S\ref{subsec:classical}). Under the assumption that $p>3$, the classification of simple Lie algebras of classical type is parallel with the classification of complex simple Lie algebras, which are completely parametrized by the Dynkin diagrams
\begin{equation*}
A_n(n\ge 1),~B_n(n\ge 2),~C_n(n\ge 3),~D_n(n\ge 4),~E_6,~E_7,~E_8,~F_4,~G_2.
\end{equation*}

A typical example of simple Lie algebra of non-classical type   is given by the \textit{Witt algebra}
\begin{equation*}
W(1;1)=\Der_k\bigl(k[X]/(X^p)\bigl), 
\end{equation*}
which is simple unless $p=2$. If $p=3$, it is isomorphic to the classical Lie algebra $\fsl_2$. However, if $p>3$, the Witt algebra $W(1;1)$ is of non-classical type. 
More generally, by considering truncated divided power polynomial rings
\begin{equation*}
A(m;\bfn)=\sum_{0\le\alpha\le\bfp^{\bfn}-\unit}k\cdot X^{(\alpha)},
\end{equation*}
where $X^{(\alpha)}=\prod_{i=1}^m``(X_i^{\alpha_i}/\alpha_i!)"$ for $\alpha=(\alpha_1,\dots,\alpha_m)\in\N^m$, one gets the \textit{$m$-th Witt algebra $W(m;\bfn)\subseteq\Der(A(m;\bfn))$ of weight $\bfn$}~(See \S\ref{subsec:witt} for the precise definition). The Witt algebras $W(m;\bfn)$ form a family of non-classical simple Lie algebras. 
By considering the differential forms
\begin{equation*}
\begin{aligned}
\omega_S&=dX_1\wedge\cdots\wedge dX_m,\\
\omega_H&=\sum_{i=1}^{r}dX_i\wedge dX_{i+r},\quad m=2r,\\
\omega_K&=dX_{2r+1}+\sum_{i=1}^r(X_idX_{i+r}-X_{i+r}dX_{i}),\quad m=2r+1,  
\end{aligned}
\end{equation*}
one gets the series of Lie subalgebras of the Witt algebras. Namely, the \textit{Special algebras} $S(m;\bfn)=\{D\in W(m;\bfn)\,|\,D(\omega_S)=0\}$, the \textit{Hamiltonian algebras} $H(2r;\bfn)=\{D\in W(2r;\bfn)\,|\,D(\omega_H)=0\}$ and the \textit{Contact algebras} $K(2r+1;\bfn)=\{D\in W(2r+1;\bfn)\,|\,D(\omega_K)\in A(2r+1;\bfn)\omega_K\}$. 
These Lie algebras $X(m;\bfn)~(X\in\{W,S,H,K\})$ are not simple in general, but the $i$-th derived subalgebra $X(m;\bfn)^{(i)}$ becomes simple for sufficiently large $i>0$. The resulting simple Lie algebras $X(m;\bfn)^{(i)}$ and the  \textit{`filtered deformations'} of them are called  \textit{Cartan type} simple Lie algebras~(cf.\ Definition \ref{def:cartan type}), which are the main objects in the present paper. In characteristic $p>3$, almost all non-classical simple Lie algebras are of Cartan type. In fact, only in characteristic $p=5$, there exists a simple Lie algebra which is neither classical nor of Cartan type. Such a Lie algebra is called of  \textit{Melikian type}~(cf.~\S\ref{subsec:p=5}). Then the classification theorem of simple Lie algebras in characteristic $p>3$ asserts that every simple Lie algebra is of  classical, Cartan or Melikian type.  
In characteristic $p=2,3$, there exist further classes of non-classical type Lie algebras~(cf.~\cite[\S 4.4]{Strade}), which we cannot deal with in the present paper.

It turns out that Theorem \ref{thm:main} for classical simple
Lie algebras is a consequence of the previous work due to the
first named author~\cite[Corollary 4.19]{ot18}~(cf.\
\S\ref{subsec:classical}). Thus, our crucial contribution
appears in the non-classical case. A difficulty in solving
Conjecture \ref{conj:PIAC A^1} for non-classical simple Lie
algebras comes from the fact that for such a Lie algebra, the
associated group scheme $\Gamma$ admits no smooth model, \emph{i.e.}\
it cannot be realized as the Frobenius kernel of a smooth
algebraic group. To handle this case, we start with a fixed faithful representation $\Gamma\hookrightarrow\GL_V$ where $\Gamma$ is a finite local non-abelian simple $k$-group scheme of height one. By choosing abelian unipotent subgroup schemes $H_1,\dots,H_n\subset \Gamma$ which generate $\Gamma$, one gets a morphism $\iota:\A_k^N\lto\GL_V^{(1)}$ for some integer $N>0$ so that the pullback $\iota^*F^{(1)}$ of the relative Frobenius homomorphism $F^{(1)}:\GL_V\lto\GL_V^{(1)}$ gives a $\GL_{V(1)}=\Ker(F^{(1)})$-torsor $P\lto\A^N_k$ whose classifying map $\pi^{\Nori}(\A^N_k)\lto\GL_{V(1)}$ has image $G$ which contains $\Gamma$, \emph{i.e.}\ $\Gamma\subseteq G$. However, it is not always true that $\Gamma=G$~(cf.\ Remark \ref{rem:AC Cartan}, Examples \ref{ex:Rumynin} and \ref{ex:counter-example in char p=3}). Therefore, more work is
needed to settle the problem. Namely, we need to control the image $G$ of the
classifying map. We do this by making a suitable choice of unipotent generators
$H_1,\dots,H_n$ of $\Gamma$. For this direction, we state a general criterion
(cf. Theorem \ref{thm:AC FLA}) for a finite local simple group scheme $\Gamma$
for being a quotient of $\pi^{\Nori}(\A_k^1)$ in terms of its Lie algebra
$\Lie(\Gamma)$ and then apply this criterion to all the mentioned Lie algebras.
In order to apply the criterion, we shall require certain \textit{filtrations}
on the Lie algebra~(cf.\ Corollary \ref{cor:AC GLA}), which fortunately
accompany many non-classical Lie algebras.
Under the assumption that $p>3$,   any non-classical simple Lie algebra $L$ has
a natural filtration indexed by the additive group $\Z$ of integers.
Actually, in many cases, one has $L\simeq\gr L$. Namely, $L$ itself admits a
$\Z$-graded structure. The remaining non-classical simple Lie algebras are
non-graded filtered \textit{Special} or \textit{Hamiltonian} algebras $L$. In
characteristic $p=5$, we can not handle the
non-graded Hamiltonian algebras $H(2r;\unit;\omega(\alpha))^{(1)}$ of
\textit{first type} with weight $\bfn=\unit$~(cf.\ Theorem \ref{thm:AC
Cartan}). That's why we have to assume that $p>5$ in Theorem
\ref{thm:main}.


Finally we explain the organization of the present paper. 
In \S\ref{sec:classical}, we mainly handle simple Lie algebras of classical type. In \S\ref{subsec:loc pi}, the definitions of Nori's fundamental group scheme and the local fundamental group scheme are recalled. The main reference is \cite{RTZ}. All our results will be stated in terms of the local fundamental group schemes. In \S\ref{subsec:lie alg}, the classification of finite local simple group schemes is recalled. We follow Viviani's paper \cite{viviani} for this aim. In \S\ref{subsec:classical}, we prove Theorem \ref{thm:main} for simple Lie algebras of classical type in characteristic $p>3$~(Corollary \ref{cor:AC classical}).

In \S\ref{sec:AC graded Lie alg}, motivated by Theorem \ref{thm:main} for simple Lie algebras of non-classical type, we give a
sufficient condition on a simple Lie algebra to have its associated group scheme
appear as a quotient of $\pi^{\Nori}(\A_k^1)$. In
\S\ref{subsec:witt vector}, we recall some facts on the additive
groups of Witt vectors and the Artin--Hasse exponential maps,
which play a fundamental role in our construction of torsors. In
\S\ref{subsec:FLA}, we recall the definitions of filtrations and
gradations on Lie algebras. We also put here some basic results
for later use. In \S\ref{subsec:criterion}, we establish
criteria Theorem \ref{thm:AC FLA} and Corollary \ref{cor:AC GLA}
for the corresponding group scheme of a given simple Lie algebra $L$ to appear as a quotient.

In \S\ref{sec:Cartan}, we recall the definition of Cartan or
Melikian type Lie algebras, and their descriptions, especially
natural filtrations or gradations of them. The main reference is
\cite{Strade}. In \S\ref{subsec:witt}, we recall the definition
of the Witt algebras. In \S\ref{subsec:gr cartan} and
\S\ref{subsec:cartan simple}, we recall the definition of Cartan
type simple Lie algebras. In \S\ref{subsec:p=5}, we recall the
definition of Melikian algebras. All the non-classical simple
Lie algebras in characteristic $p>3$ are listed in
\S\ref{sec:Cartan}. Here we are looking at these algebras with a
view towards applications of Theorem \ref{thm:AC FLA} and Corollary
\ref{cor:AC GLA}, so that the group schemes corresponding to
most of
these algebras will eventually appear as quotients of
$\pi^{\Nori}(\A_k^1)$.

In \S\ref{sec:application}, we prove Theorem \ref{thm:main}. As the classical case is handled in \S\ref{subsec:classical}, it suffices to prove the theorem for non-classical type simple Lie algebras. For them, 
we will apply the criterions Theorem \ref{thm:AC FLA} and Corollary \ref{cor:AC GLA}. In fact, the theorem is essentially proved in \S\ref{sec:Cartan}. Here we just apply classification theorems for non-classical simple Lie algebras. In \S\ref{subsec:KS conj}, we recall the classification theorems in characteristic $p>3$, especially the statement of the generalized Kostrikin--Shafarevich conjecture.  In \S\ref{subsec:AC non classical}, we complete the proof of Theorem \ref{thm:main}.


\begin{ack}
The main result in the present paper is a partial answer to the question given
by Dmitriy Rumynin after the first named author's talk at Max-Planck Institut
f\"ur Mathematik in Bonn,  September 2018. The authors would like to thank him
for suggesting considering Cartan type Lie algebras as a test for the
generalized Abhyankar's conjecture. The first named author would like to thank
Takuya Yamauchi for arranging the talk. The authors also thank them for giving
helpful comments on the first draft of the present paper. The authors would
like to thank Takao Yamazaki, Madhav Nori, Tomoyuki Abe and Jo\~ao Pedro dos
Santos for having fruitful discussions and suggestions the authors received.
Finally, the authors would like to thank the anonymous referee for his or her
valuable comments which are very useful in improving the manuscript.

The first named author was supported by JSPS Grant-in-Aid for JSPS Research Fellow, Grant Number 19J00366. The second author was supported by GNSAGA of INdAM. The third author is supported by  the Research Grants Council (RGC) of the Hong Kong SAR China (Project No.\ CUHK 14301019).
\end{ack}


\section*{\centering Notations and Conventions}
\label{s:notation}

\begin{enumerate}
\renewcommand{\labelenumi}{(\arabic{enumi})}
    \item  In the present paper, a Lie algebra over a field $k$ is always assumed to be \textit{finite dimensional}. 
Moreover, a \textit{simple} Lie algebra over $k$ is always assumed to be
\textit{non-abelian}. These conventions are also adopted for the restricted Lie algebras. 
   \item  By a \textit{simple Lie algebra} we mean a Lie algebra that is
       non-abelian and contains no nonzero proper ideals. 
   \item Let $k$ be a field. By a \textit{simple $k$-group scheme}, we mean a
       linear algebraic group $G$ over  $k$ that contains no nonzero proper
       closed normal subgroup schemes.
    \item Let $k$ be a field, and let $i\colon H\rightarrow G$ be a
       homomorphism of $k$-group schemes. We say that $i$ is
       \textit{injective} if $i$ a monomorphism of sheaves of groups, \emph{i.e.} $i(T)\colon H(T)\to G(T)$
       is injective for all $k$-schemes $T$. If $H$ and $G$ are affine then the
       map $i$ is injective if and only if it is a closed
       embedding~(cf.\ \cite[Chapter 15, \S3]{wa79}). 

   \item Given a simple restricted Lie algebra $L$ over a field $k$, we
       denote by $\mathfrak{G}(L)$ the corresponding finite local $k$-group
       scheme over $k$ under the correspondence \eqref{eq:scheme vs restricted
       Lie algebra}. By construction $L\simeq \Lie( \mathfrak{G}(L))$.
       
   \item Let $L$ be a Lie algerba (resp. restricted Lie algebra) over a field $k$, and let
        $V\subseteq L$ be a subset of $L$. We denoted by $\langle
        V\rangle_{\Lie}$ (resp. $\langle
        V\rangle_{p-\Lie}$) the \textit{Lie subalgebra {\rm (}resp.  restricted Lie
            subalgebra{\rm)}
        generated by $V$}. This is the intersection of all Lie
        subalgebras (resp. restricted Lie subalgebras) of $L$ containing $V$ and, in particular, it is the smallest Lie subalgebra
        (resp. restricted Lie
        subalgebra) of $L$
        containing $V$.
    \item Let $G$ be an affine group scheme. Let $V\coloneqq
        \{T_i\to G\}_{i\in I}$ be a collection of points of $G$, where each $T_i$
        is a $k$-scheme. We denote by $\langle V\rangle$ or $\langle
        T_i\to G\rangle_{i\in I}$ the \textit{subgroup of $G$ generated
        by $V$}. This is the intersection of all affine subgroups $H\subseteq G$ containing
        all the points in $V$,   \emph{i.e.} $T_i\to G$ belongs to
        $H(T_i)\subseteq G(T_i)$ for each $i\in I$. In particular, it is the
        smallest affine subgroup of $G$ with this property.
\end{enumerate}

\section{Reduction to the non-classical case}\label{sec:classical}

\subsection{The local fundamental group scheme}\label{subsec:loc pi}

In this subsection, we briefly recall the definition of Nori's \textit{fundamental group scheme}~(cf.\ \cite{no76}\cite[Chapter II]{no82}) and its \textit{local} variant~(cf.\ \cite[\S2]{RTZ}). 

Let $k$ be a field and $X$ a geometrically connected and reduced scheme of finite type over $k$ together with a $k$-rational point $x\in X(k)$. Then there exists a profinite $k$-group scheme $\pi^{\Nori}(X,x)$, called the \textit{fundamental group scheme} of $(X,x)$, such that for any finite $k$-group scheme $G$, the set $\Hom(\pi^{\Nori}(X,x),G)$ of $k$-homomorphisms $\pi^{\Nori}(X,x)\lto G$ is naturally in bijection with the set  of isomorphism classes of pointed $G$-torsors $(P,p)\lto (X,x)$. The construction $(X,x)\longmapsto \pi^{\Nori}(X,x)$ gives a covariant functor from the category of pointed geometrically connected and reduced schemes of finite type over $k$ into the category of profinite $k$-group schemes.  

Suppose that $k$ is a perfect field of characteristic $p>0$.
Recall that a finite $k$-group scheme $G$ is called
\textit{local} if it is connected as a scheme. Let $X$ be a reduced scheme over $k$. According to
\cite[\S2]{RTZ}, there exists a \textit{pro-finite local}~(cf.\
\cite[Definition 2.1]{RTZ}) $k$-group scheme $\pi^{\loc}(X)$, which we call the \textit{local fundamental group scheme} of $X$, such that it pro-represents the functor
\begin{equation*}
G\longmapsto H^1_{\fppf}(X,G)
\end{equation*}
from the category of finite local $k$-group schemes to the category of sets, where $H^1_{\fppf}(X,G)$ is nothing but the set of isomorphism classes of $G$-torsors over $X$ (cf. \cite[Proposition 2.13]{RTZ}). Namely, for any finite local $k$-group scheme $G$, there exists a canonical bijection of the sets
\begin{equation*}
\Hom_k(\pi^{\loc}(X),G)\xrightarrow{~\simeq~}H^1_{\fppf}(X,G).
\end{equation*} 
Note that a $G$-torsor $P\lto X$ corresponds to a surjective homomorphism $\pi^{\loc}(X)\twoheadrightarrow G$ if and only if there exists no strict  subgroup scheme $H\subsetneq G$ such that $P$ is reduced to an $H$-torsor $Q\lto X$, \emph{i.e.}\ $P\simeq Q\wedge^H G$.

\begin{definition}\label{def: local quotients}
Let $X$ be a reduced scheme over a perfect field $k$. We set 
\[
\pi^\loc_A(X)=\{\text{finite $k$-group schemes which appear as quotients of }\pi^\loc(X)\}
\]
\end{definition}

In the particular case when $X$ is a geometrically connected and reduced scheme of finite type over a perfect field $k$ of characteristic $p>0$ together with a $k$-rational point $x\in X(k)$, it turns out that $\pi^{\loc}(X)$ is canonically isomorphic to the maximal local quotient of the fundamental group scheme $\pi^{\Nori}(X,x)$. In particular, the maximal local quotient of $\pi^{\Nori}(X,x)$ does not depend on the choice of the rational point $x$. Moreover
\[
\pi^\loc_A(X)=\{\text{finite \emph{local} $k$-group schemes which appear as quotients of }\pi^\Nori(X,x)\}
\]

\subsection{Classification of finite local simple group schemes}\label{subsec:lie alg}

Let $k$ be a field. Let $G$ be an affine $k$-group scheme of
finite type. Recall that the Lie algebra $\Lie(G)$ of $G$ is
defined to be the $k$-vector space of left-invariant derivations
$D:k[G]\lto k[G]$ of the coordinate ring $k[G]$~(cf.\
\cite[\S12.1]{wa79}). If $k$ is of characteristic $p>0$, then
the associated Lie algebra $\Lie(G)$ has a restricted structure
$[p]:\Lie(G)\lto\Lie(G)$ induced by the $p$-th power map of
derivations $D\longmapsto D^p$~(cf.~\cite[\S12.1]{wa79}). This
correspondence $G\longmapsto (\Lie(G),[p])$ induces the
equivalence of categories between the category of finite local
$k$-group schemes of height one and the category of finite
dimensional restricted Lie algebras over $k$~(cf.\ \cite[Chapter II, \S 7, n$^\textup{o}$4, 4.1]{dg70}). In fact, for a finite local $k$-group scheme $G$
of height one with $\fg=\Lie(G)$, the dual associative algebra
$k[G]^{\vee}$ of $k[G]$ is canonically isomorphic to the
restricted enveloping algebra of $\fg$, \emph{i.e.}\
$U^{[p]}(\fg)\xrightarrow{~\simeq~}k[G]^{\vee}$. Thus the
associative algebra $k[G]^{\vee}$ is generated by the image
of a basis of $\fg$ and $\Dim_k k[G]=\Dim_k U^{[p]}(\fg)=p^{\Dim_k\fg}$~(cf.\ \cite[Chapter I, 7.10]{ja03}). 
Thus, it follows that a homomorphism $G\lto H$ of finite local $k$-group schemes of height one is surjective (respectively injective) if and only if the induced map between the Lie algebras $\Lie(G)\lto\Lie(H)$ is surjective (respectively injective). As the functor $G\longmapsto\Lie(G)$ is left exact from the category of affine group schemes of finite type over $k$ into the category of finite dimensional Lie algebras over $k$ ~(cf.\ \cite[Chapter II, \S 4, n$^\textup{o}$1, 1.5]{dg70}), this immediately implies that a sequence of homomorphisms of finite local $k$-group schemes of height one
\begin{equation*}
1\lto G'\lto G\lto G''\lto 1
\end{equation*}
is exact if and only if the induced sequence of homomorphisms of restricted Lie algebras
\begin{equation*}
0\lto\Lie(G')\lto\Lie(G)\lto\Lie(G'')\lto 0
\end{equation*}
is exact. In particular, a finite local non-abelian $k$-group scheme $G$ of
height one is simple if and only if the associated restricted Lie algebra
$\Lie(G)$ is simple~(cf.\ \cite[\S2, Corollary 1]{viviani}).

Let $L$ be a Lie algebra over $k$. 
Let $\Der(L)$ denote the subspace of $\fgl_L$ consisting of
\textit{derivations} of $L$~(cf.\ \cite[Chapter 10, a]{mi17}). Then $\Der(L)$ is closed under the Lie bracket of $\fgl_L$ and the $p$-th power map of $L$. Namely $\Der(L)$ is a restricted Lie subalgebra of $\fgl_L$. 
Let the following map
\begin{equation}\label{eq:adjoint}
\ad:L\lto\Der(L)\,;\,x\longmapsto\ad(x)
\end{equation}
denote the  \textit{adjoint representation} of $L$~(cf.\ \cite[Chapter 10, a]{mi17}). 
If $L$ has a restricted structure $[p]:L\lto L$, then the adjoint representation is compatible with the restricted structures, \emph{i.e.}\ $\ad(x^{[p]})=\ad(x)^p$ for any $x\in L$.

\begin{definition}(cf.\ \cite[Definition 1.1.2]{Strade})\label{def:p-env}
Let $G$ be a restricted Lie algebra over $k$ and $i\colon L\lto G$ be an injective
homomorphism of Lie algebras. We denote by
$L_{[p],i}$ the 
smallest restricted Lie subalgebra of $G$ containing $i(L)$ and call it the
\textit{$p$-envelope} of $L$ in $G$.

If $L$ is a centerless Lie algebra and $i=\ad \colon L\lto G=\Der(L)$ is the adjoint
representation \eqref{eq:adjoint} (which is injective because $L$ is centerless) we set $L_{[p]} = L_{[p],\ad}$.
\end{definition}

\begin{thm}{\rm (cf.\ \cite[\S4, Theorem 4]{viviani})}\label{thm:viviani} 
There exists a bijection between the set of isomorphism classes of simple Lie algebras over $k$ and the set of isomorphism classes simple restricted Lie algebras over $k$. 
\begin{equation*}
\Bigl\{
~\text{simple Lie algebras over $k$}~
\Bigl\}~
\overset{1:1}{\longleftrightarrow}
~\Bigl\{
~\text{simple restricted Lie algebras over $k$}~
\Bigl\}.
\end{equation*}
More precisely, for any simple Lie algebra $L$ over $k$, let $L_{[p]}$ be the $p$-envelope of $L$ in the adjoint representation $L\hookrightarrow\Der(L)$. Then $L_{[p]}$ is a simple restricted Lie algebra over $k$.  Conversely, for any simple restricted Lie algebra $L$ over $k$, the derived subalgebra $[L,L]$ becomes a simple Lie algebra over $k$. 
\end{thm}

Thus, the classification of finite simple local $k$-group scheme of height one is equivalent to the classification of  simple Lie algebras over $k$. Namely, there exists a bijection between the set of isomorphism classes of simple Lie algebras over $k$ and the set of isomorphism classes of finite local non-abelian simple $k$-group schemes. 
\begin{equation*}
\Bigl\{
~\text{simple Lie algebras over $k$}~
\Bigl\}~
\overset{1:1}{\longleftrightarrow}
~\Bigl\{
~\text{finite local non-abelian simple $k$-group schemes}~
\Bigl\}.
\end{equation*}

\begin{rem}
Let $k$ be a perfect field of characteristic $p>0$ with $\overline{k}$ an algebraic closure of $k$. Let $L$ be a Lie algebra over $k$. If $L\otimes_k\overline{k}$ is simple, then so is $L$. 
\end{rem}

\subsection{The conjecture for classical simple Lie algebras}\label{subsec:classical}

The first class of simple Lie algebras which we handle is simple Lie algebras of \textit{classical type}~(cf.~\cite[\S4.1]{Strade}). Recall that an algebraic group $G$ over a field $k$ is said to be \textit{almost simple} if it is a connected semisimple algebraic group such that the quotient $G/Z$ of $G$ by the center $Z$ is simple~(cf.\ \cite[5.6.6]{Poonen}). Note that the center $Z$ of a connected semisimple algebraic group is always a finite group scheme of multiplicative type~(cf.\ \cite[Corollary 5.6.11]{Poonen}). A homomorphism $G\lto G'$ of connected algebraic groups over a field $k$ is said to be a \textit{central isogeny} if it is a surjective homomorphism whose kernel is finite and is contained in the center of $G$. If $G$ is a connected semisimple algebraic group over $k$, then there exists a semisimple simply connected algebraic group $G^{\rm sc}$ which is unique up to isomorphism such that $G^{\rm sc}/H\simeq G$ for some subgroup scheme $H\subseteq Z(G^{\rm sc})$ of the center of $G^{\rm sc}$~(cf.\ \cite[Proposition 5.6.21(a)]{Poonen}). We call the central isogeny $G^{\rm sc}\lto G$ a \textit{simply connected cover} of $G$. Any two connected semisimple algebraic groups $G$ and $G'$ are said to be \textit{centrally isogeneous} if $G^{\rm sc}\simeq G'^{\rm sc}$ as a $k$-group scheme. Then, according to \cite[Theorem 5.6.26]{Poonen}, the centrally isogeneous classes of almost simple algebraic groups over a separably closed field $k$ are completely parametrized by the Dynkin diagrams of types
\begin{equation}\label{eq:dynkin}
A_n(n\ge 1),~B_n(n\ge 2),~C_n(n\ge 3),~D_n(n\ge 4),~E_6,~E_7,~E_8,~F_4,~G_2.
\end{equation}
In the case where $k=\C$, the classification of simple Lie algebras is equivalent to the classification of centrally isogeneous classes of almost simple algebraic groups over $\C$. Namely, the list of the Dynkin diagrams (\ref{eq:dynkin}) gives a complete classification of simple complex Lie algebras. More precisely, for any simple complex algebra $L$, there exists an almost simple algebraic group $G$ over $\C$ such that $L\simeq\Lie(G)$. Moreover, for any almost simple algebraic groups $G$ and $G'$ over $\C$, the associated simple Lie algebras are isomorphic $\Lie(G)\simeq\Lie(G')$ to each other if and only if $G$ and $G'$ have the same Dynkin type $\mathcal{D}$, or equivalently they are centrally isogeneous. Here, note that 
in the characteristic $0$ case, any central isogeny $G\lto G'$ of connected algebraic groups is always \'etale, hence it induces an isomorphism $\Lie(G)\xrightarrow{~\simeq~}\Lie(G')$ between the associated Lie algebras.

Let $L_{\C}$ be a semisimple Lie algebra over $\C$ with $H_{\C}\subset L_{\C}$ a Cartan subalgebra and $\Phi$ the root system. Choose a positive system $\Phi^+\subset\Phi$ and denote by $\Delta=\{\alpha_1,\dots,\alpha_l\}$ the corresponding set of simple roots~(cf.\ \cite[II Chapter 1, 1.5]{ja03}). 
Then $L_{\C}$ admits a \textit{Chevalley basis} $\{x_{\alpha}\}_{\alpha\in\Phi}\cup\{h_i\}_{i=1}^{l}$~(cf.~\cite[Theorem 4.1.1]{Strade}), and the $\Z$-span $L_{\Z}$ of the Chevalley basis gives a $\Z$-Lie subalgebra of $L_{\C}$. For any field $k$, we get a Lie algebra
\begin{equation*}
L_k\Def L_{\Z}\otimes_{\Z}k, 
\end{equation*}
which is determined uniquely up to isomorphism by the complex  
simple Lie algebra $L_{\C}$. The Lie algebra $L_k$ is called the \textit{Chevalley algebra} over $k$ associated with the semisimple complex Lie algebra $L_{\C}$. The Chevalley algebra $L_k$ can be always interpreted as the Lie algebra of a certain algebraic group over $k$ in the following way. Let us take a connected semisimple algebraic group $G_{\C}$ over $\C$ such that $\Lie(G_{\C})\simeq L_{\C}$. Then there exists an affine algebraic $\Z$-group scheme $G$ satisfying the following condition~(cf.\ \cite{ab69}\cite{vp96}).
\begin{enumerate}
\renewcommand{\labelenumi}{(\roman{enumi})}
\item $G\otimes_{\Z}\C$ is isomorphic to $G_{\C}$ as an algebraic group over $\C$. 
\item For any algebraically closed field $k$, $G_k\Def G\otimes_{\Z}k$ is a connected semisimple algebraic group  over $k$ which is split over the prime field  and has the same Dynkin type as $G_{\C}$. Moreover, $G_k$ is simply connected (respectively adjoint) if and only if $G_{\C}$ is simply connected (respectively adjoint). 
\end{enumerate}
The $\Z$-group scheme $G$ is determined uniquely up to isomorphism by the semisimple algebraic group $G_{\C}$, and is called the \textit{Chevalley--Demazure group scheme} associated with $G_{\C}$. For any field $k$, we simply call $G_k$ the \textit{Chevalley group} over $k$ associated with the complex semisimple algebraic group $G_{\C}$. In the case where $k$ is of characteristic $p\neq 2,3$, the Lie algebra of the Chevalley group $G_k$ is isomorphic to the Chevalley algebra $L_k$. 

\begin{lem}{\rm (cf.\ \cite[\S2.3 Lemme 4]{mathieu})}\label{lem:mathieu}
With the above notation, we have $L_{\Z}\otimes_{\Z}\Z[\frac{1}{6}]=\Lie(G)\otimes_{\Z}\Z[\frac{1}{6}]$ in $L_{\C}=\Lie(G)\otimes_{\Z}\C$.
\end{lem}

In particular, if $k$ is of characteristic $p>3$, then the Chevalley algebra $L_k$ admits a restricted structure, which is induced by the one of $\Lie(G_k)$. In terms of the fixed Chevalley basis, it can be uniquely characterized by the rule
\begin{equation*}
(x_{\alpha}\otimes 1)^{[p]}=0~~\text{and}~~(h_i\otimes 1)^{[p]}=h_i\otimes 1.
\end{equation*} 
Moreover, the finite local $k$-group scheme $\mathfrak{G}(L_k)$~\textup{(cf.\ \nameref{s:notation})} associated with the Chevalley algebra $L_k$ is nothing but the first Frobenius kernel $G_{k(1)}\Def\Ker(F^{(1)}:G_k\lto G_k^{(1)})$ of the Chevalley group $G_{k}$.

The Chevalley algebra $L_k$ is not simple even if $L_{\C}$ is simple. In fact, it turns out that if $k$ is of characteristic $p>3$, it is simple unless $L_{\C}$ is of type $A_l$ with $l\equiv -1$ mod $p$, in which $L$ has a one-dimensional center $Z=k\cdot (h_1+2h_2+\cdots+lh_l)$ so that $L/Z\simeq\mathfrak{psl}_{l+1,k}$ is a simple Lie algebra~(cf.~\cite[Lemme 2]{mathieu}\cite[\S4.1]{Strade}). 
Note that 
\begin{equation*}
(h_1+2h_2+\cdots+lh_l)^{[p]}=h_1^{[p]}+2^ph_2^{[p]}+\cdots+l^ph_l^{[p]}=h_1+2h_2+\cdots+lh_l, 
\end{equation*}
and hence $Z$ is a restricted Lie subalgebra of $L$. Thus, the quotient $L/Z$ also has a restricted structure. 

\begin{definition}(cf.\ \cite[D\'efinition 3]{mathieu}\cite[\S4.2]{Strade})\label{def:classical type}
Let $k$ be an algebraically closed field of characteristic $p>3$. A simple Lie algebra over $k$ is said to be of \textit{classical type} if it is isomorphic to either the Lie algebra 
\begin{equation*}
\mathfrak{psl}_{n+1,k}~(n\equiv -1~\text{mod}~p)
\end{equation*}
or the Chevalley algebra $L_k$ associated with a simple complex Lie algebra $L_{\C}$ of type 
\begin{equation*}
A_n~(n\not\equiv -1~\text{mod}~p)
,~B_n,~C_n,~D_n,~E_6,~E_7,~E_8,~F_4,~G_2.
\end{equation*}
\end{definition}

Now let us prove Theorem \ref{thm:main} for the classical simple Lie algebras. We deduce it from the following result.

\begin{thm}{\rm (cf.~\cite[Corollary 4.19]{ot18})}\label{thm:AC Sigma}
Let $\Sigma$ be a split semisimple simply connected algebraic group over a perfect field $k$ of characteristic $p>0$. Then 
\begin{equation*}
\Sigma_{(r)}\Def\Ker\bigl(\Sigma\xrightarrow{F^{(r)}}\Sigma^{(r)}\bigl) \in \pi^\loc_A(\A^1_k)
\end{equation*}
for any integer $r>0$. 
\end{thm}

This is an analogue of Nori's Theorem \cite{no94}~(cf.\ Example \ref{ex:quasi-p}).  
The idea of the proof is the same as the one due to Laumon \cite[\S 3.2]{se92}\cite[Theorem 3.5]{ha18}. It suffices to establish a Bertini type result for finite local torsors of affine spaces $\A^n_k$. Indeed, one can prove the following result. For the proof, see \cite[Proof of Theorem 4.17]{ot18}. This will be used again in the proof of Theorem \ref{thm:AC FLA}.

\begin{lem}{\rm (cf.\ \cite[Theorem 4.17]{ot18})}\label{lem:bertini}
Let $k$ be a perfect field of characteristic $p>0$. 
Let $n\ge 2$ be an integer. Let $\Gamma$ be a finite local $k$-group scheme of height one and $\pi^{\loc}(\A^n_k)\twoheadrightarrow \Gamma$ a surjective $k$-homomorphism. Then there exists a closed immersion $\A^{n-1}_k\hookrightarrow\A^n_k$ such that the composition map
\begin{equation*}
\pi^{\loc}(\A^{n-1}_k)\lto\pi^{\loc}(\A^n_k)\lto \Gamma
\end{equation*}
is surjective. 
\end{lem}

The following is an immediate consequence of Theorem \ref{thm:AC Sigma}.

\begin{cor}\label{cor:AC chevalley}
Let $L_k$ be a Chevalley algebra over a perfect field $k$ of
characteristic $p>3$ associated with a semisimple complex Lie
algebra $L_{\C}$. Then
$\mathfrak{G}(L_k)\in \pi^\loc_A(\A^1_k)$.
\end{cor}

\begin{proof}
As $L_{\C}$ is a direct sum of complex simple Lie algebras, there exists a
semisimple simply connected algebraic $G_{\C}$ group over $\C$ such that
$\Lie(G_{\C})\simeq L_{\C}$. Then the associated Chevalley group $G_k$ over $k$
is a split semisimple simply connected algebraic group. As $p>3$, by Lemma
\ref{lem:mathieu}, we have $\Lie(G_k)\simeq L_k$. In other words, $\mathfrak{G}(L_k)$ is isomorphic to the first Frobenius kernel of the Chevalley group $G_k$. As $G_k$ is simply connected, the assertion is thus immediate from Theorem \ref{thm:AC Sigma}. This completes the proof. 
\end{proof}

In particular, we can conclude the following.

\begin{cor}\label{cor:AC classical}
Let $k$ be an algebraically closed field of characteristic $p>3$. Then $\mathfrak{G}(L_{[p]})\in \pi^\loc_A(\A^1_k)$~\textup{(cf.\ Definition \ref{def:p-env})} for any
simple Lie algebra $L$ over $k$ of classical type.
\end{cor}

\begin{proof}
By Definition \ref{def:classical type}, any simple Lie algebra over $k$ of classical type is a quotient of the Chevalley algebra over $k$ associated with a complex simple Lie algebra $L_{\C}$. Therefore, the assertion is immediate from Corollary \ref{cor:AC chevalley}. 
\end{proof}

\begin{rem}\label{rem:classical p=3}
If $k$ is of characteristic $p=3$, a Chevalley algebra $L_k$  is simple unless it is of type 
\begin{equation*}
A_n~(n\equiv -1~\text{mod}~3),~E_6,~G_2,
\end{equation*}
in which case each of them has a unique ideal~(cf.\ \cite[\S4.4]{Strade}). The simple quotients of the Chevalley algebras over an algebraically closed field of characteristic $p=3$ are also called \textit{classical}.
\end{rem}


\section{A criterion for $\Z$-graded Lie algebras}\label{sec:AC graded Lie alg}

\subsection{The group of Witt vectors}\label{subsec:witt vector}

In this subsection, we collect some foundations on the group of Witt vectors from \cite[\S1.5]{sobaje}. 
Let $k$ be a field of characteristic $p>0$. For an integer $n\ge 1$, let $W_n$ denote the $k$-group scheme of Witt vectors of length $n$ with standard coordinates $(T_0,T_1,\dots,T_{n-1}):W_n\xrightarrow{~\simeq~}\A^n_k$~(cf.\ \cite[\S VII.8]{serre GTM}). The group scheme $W_n$ is an abelian  unipotent algebraic group over $k$. Recall that if $n=1$, then $W_1=\G_{a,k}$, the additive group, and for  $n>1$, it fits into the following exact sequence of abelian algebraic groups
\begin{equation*}
1\lto W_{n-1}\lto W_n\lto\G_{a,k}\lto 1.
\end{equation*}
On the other hand, via the coordinates $(T_0,T_1,\dots,T_{n-1})$, its Lie algebra $\Lie(W_n)$ can be identified with
\begin{equation*}
\Lie(W_n)=\sum_{i=0}^{n-1}k\cdot\frac{\partial}{\partial T_i}
\end{equation*} 
on which the restricted structure is given by
\begin{equation*}
\Bigl(\frac{\partial}{\partial T_i}\Bigl)^{[p]}=
\begin{cases}
\frac{\partial}{\partial T_{i+1}}&\text{if $0\le i\le n-2$},\\
0&\text{if $i=n-1$}.
\end{cases}
\end{equation*}
Notice that the first Frobenius kernel $W_{n(1)}\Def\Ker(F^{(1)}:W_n\lto W_n^{(1)})$ is the Cartier dual of the group scheme $\alpha_{p^n}=\Ker(F^{(n)}:\G_{a,k}\lto\G_{a,k}^{(n)})$.  

Next recall that the \textit{Artin--Hasse exponential series} is the formal power series with coefficients in $\Q$ defined to be
\begin{equation*}
E_p(T)\Def\exp\Biggl(\sum_{j=0}^{\infty}\frac{T^{p^j}}{p^j}\Biggl)\in\Q[[T]]
\end{equation*}  
As is well-understood, we have $E_p(T)\in\Z_{(p)}[[T]]$~(cf.\ \cite[\S V.16]{serre GTM}). Thus, by taking reduction modulo $p$, we obtain
\begin{equation}\label{eq:e_p}
e_p(T)\Def E_p(T)~\text{mod $p$}~\in\F_p[[T]]\subseteq k[[T]].
\end{equation}

The following fact will play a fundamental role in the proof of Theorem \ref{thm:AC FLA}.

\begin{thm}\label{thm:f_X}
Let $V$ be a finite dimensional $k$-vector space. Let $X\in\fgl_V$ be an endomorphism of $V$ such that $X^{p^n}=0$ and $X^{p^{n-1}}\neq 0$ for some integer $n\ge 1$. Then the map
\begin{equation}\label{eq:f_X}
f_X:W_n\lto\GL_V
\end{equation}
defined by
\begin{equation*}
(a_0,a_1,\dots,a_{n-1})\longmapsto e_p(a_0X)e_p(a_1X^p)\cdots e_p(a_{n-1}X^{p^{n-1}})
\end{equation*}
gives a closed and injective homomorphism of $k$-group schemes. Moreover, the induced map $df_X:\Lie(W_n)\lto\fgl_V$ on the Lie algebras is given by
\begin{equation}\label{eq:df_X}
df_X\Bigl(\frac{\partial}{\partial T_i}\Bigl)=X^{p^{i}} 
\end{equation}
for $0\le i\le n-1$.
\end{thm}

\begin{proof} Let $\bar{k}$ be an algebraic closure of $k$, and let
    $X\otimes_k\bar{k}$ be the base change of $X\in\fgl_V$ to
    $\fgl_{V\otimes_k\bar{k}}$. Then we have
    $f_{X\otimes_k\bar{k}}=f_X\times_k\bar{k}$ and
    $df_{X\otimes_k\bar{k}}=df_X\otimes_k\bar{k}$. Thus the claim is reduced
    to the case when $k=\bar{k}$, \emph{i.e.} when $k$ is
    algebraically closed. The latter follows from
    \cite[Theorem 7.4]{Proud}.
\end{proof}
\todo[inline]{Is the above argument right?}

\begin{rem}\label{rem:f_X}
Let $V$ be a finite dimensional $k$-vector space. Let $X_1,X_2\in\fgl_V$ be endomorphisms of $V$ such that, for each $i=1,2$, there exists an integer $n_i>0$ satisfying $X_i^{p^{n_i}}=0$ and $X_i^{p^{{n_i}-1}}\neq 0$. If $[X_1,X_2]=0$ in $\fgl_V$, then, the maps $f_{X_1}:W_{n_1}\lto\GL_V$ and $f_{X_2}:W_{n_2}\lto\GL_V$ defined in Theorem \ref{thm:f_X} commutes with each other. This is valid because, by the equation (\ref{eq:e_p}), for any $k$-algebra $A$ and any $a_i\in W_{n_i}(A)$ for $i=1,2$, each images $f_{X_i}(a_i)$ can be described as a power series of $X_i$ with coefficients in $A$, and hence $f_{X_1}(a_1)$ and $f_{X_2}(a_2)$ commute with each other if $[X_1,X_2]=0$. 
Therefore, in particular, it follows that the product
\begin{equation*}
W_{n_1}\times W_{n_2}\lto\GL_V\,;\,(t,s)\longmapsto f_{X_1}(t)\cdot f_{X_2}(s)
\end{equation*} 
gives a homomorphism of $k$-group schemes. 
\end{rem}

\subsection{Filtered Lie algebras}\label{subsec:FLA}

Let $k$ be a field of characteristic $p>0$.

\begin{definition}\label{def:GRLA}
A $\Z$-\textit{graded Lie algebra} over $k$ is a Lie algebra $L$ whose underlying vector space over $k$ has a grading indexed by $\Z$,
\begin{equation*}
L=\bigoplus_{i\in\Z}L_i
\end{equation*}
satisfying
\begin{equation*}
[L_i,L_j]\subseteq L_{i+j}
\end{equation*}
for any $i,j\in\Z$. For a finite dimensional graded Lie algebra $L$, the \textit{depth} of $L$, which we denote by $s(L)$, means the minimal number $s\ge 0$ such that $L_{-j}=0$ for any $j> s$, and the \textit{height} of $L$, which we denote by $h(L)$, means the minimal number $h\ge 0$ such that $L_{i}=0$ for any $i> h$. A $\Z$-\textit{graded restricted Lie algebra} $L$ is a $\Z$-graded Lie algebra $L=\bigoplus_{i\in\Z}L_i$ which has a restricted structure $[p]:L\lto L$ satisfying
\begin{equation*}
L_i^{[p]}\subseteq L_{ip}
\end{equation*}
for any $i\in\Z$, where $L^{[p]}$ is the image of the map $[p]:L\lto L$.  

For each integer $j\in\Z$, we set
\begin{equation*}
L_{\le j}\Def\bigoplus_{i\le j}L_i.  
\end{equation*}
\end{definition}

\begin{ex}
Let $L=\fsl_2$ be the restricted Lie algebra associated with the special linear algebraic group $\SL_2$ over a field $k$ of characteristic $p>2$. Then the root decomposition of $L$,
\begin{equation*}
L=L_{-2}\oplus L_0\oplus L_2
\end{equation*}
endows $L$ with a structure of $\Z$-graded restricted Lie algebra. 
\end{ex}

\begin{ex}
Let $L=\Der_k(k[X]/(X^p))$ be the Lie algebra of derivations of the truncated polynomial ring $k[X]/(X^p)$ with restricted structure given by the $p$-th power map of derivations $[p]:L\lto L;D\mapsto D^p$. If we put $e_i=X^{i+1}\frac{d}{dX}$ for $-1\le i\le p-2$, then we have
\begin{equation*}
L=\bigoplus_{i=-1}^{p-2}L_i\quad\text{with}\quad L_i=k\cdot e_i.
\end{equation*}
Moreover,
\begin{equation*}
[e_i,e_j]=(j+i+1)e_{i+j},\quad e_i^{[p]}=0~\text{for}~i\neq 0~\text{and}~e_0^{[p]}=e_0. 
\end{equation*}
The Lie algebra $L$ has a structure of $\Z$-graded restricted Lie algebra over $k$.  
\end{ex}

\begin{definition}\label{def:FLA}
A \textit{filtered Lie algebra} over $k$ is a Lie algebra $L$ together with a descending filtration
\begin{equation*}
L=\bigcup_{i\in\Z} L_{(i)}\quad\text{with}\quad L_{(i)}\supseteq L_{(i+1)}
\end{equation*} 
satisfying $L=L_{(-i)}$ for $i\gg 0$, $L_{(j)}=0$ for $j\gg 0$ and $[L_{(i)},L_{(j)}]\subseteq L_{(i+j)}$ for any $i,j\in\Z$. For a filtered Lie algebra $L$ over $k$, the \textit{depth} of $L$, which we denote by $s(L)$, means the minimal number $s\ge 0$ such that $L_{(-j)}=L$ for any $j\ge s$, and the \textit{height} of $L$, which we denote by $h(L)$, means the minimal number $h\ge 0$ such that $L_{(i)}=0$ for any $i> h$. 
For a filtered Lie algebra $L$ over $k$, one can associate the $\Z$-graded Lie algebra $\gr L$ in the sense of Definition \ref{def:GRLA} by setting
\begin{equation*}
\gr_iL\Def L_{(i)}/L_{(i+1)}\quad\text{and}\quad\gr L\Def\bigoplus_{i\in\Z}\gr_i L
\end{equation*}
(cf.\ \cite[4.3.3]{Bois}). More precisely, let us denote by $\gr_i:L_{(i)}\twoheadrightarrow L_{(i)}/L_{(i+1)}=\gr_iL$ the natural projection map for each $i\in\Z$, and set $\overline{x}\Def\gr_{\nu(x)}(x)$ for each nonzero element $0\neq x\in L$, where $\nu(x)\Def\max\{k\in\Z\,|\,x\in L_{(k)}\}$. Then, for any nonzero elements $0\neq x,y\in L$, we have
\begin{equation}\label{eq:lie bracket gr L}
[\overline{x},\overline{y}]=\gr_{\nu(x)+\nu(y)}([x,y])=
\begin{cases}
\overline{[x,y]}&\text{if $[x,y]\not\in L_{(\nu(x)+\nu(y)+1)}$},\\
0&\text{if $[x,y]\in L_{(\nu(x)+\nu(y)+1)}$},
\end{cases}
\end{equation}
in $\gr L$. The notions of depth and height are also compatible, \emph{i.e.}\  $s(L)=s(\gr L)$ and $h(L)=h(\gr L)$. Conversely, one can consider any $\Z$-graded Lie algebra $L=\oplus_{i\in\Z}L_i$ over $k$ as a filtered Lie algebra by setting
\begin{equation*}
L_{(i)}\Def \bigoplus_{j\ge i}L_j. 
\end{equation*}
\end{definition}

\begin{lem}\label{lem:GLA vs FLA}
Let $L$ be a Lie algebra over $k$ 
\begin{enumerate}
\renewcommand{\labelenumi}{(\arabic{enumi})}
\item Let $L$ be a filtered Lie algebra over $k$, and let $S\subseteq L$
    be a subset. If the image of $S$ in $\gr L$ generates $\gr L$, then $L=\langle
    S\rangle_{\Lie}$.
\item Let $U^-,U^+\subseteq L$ be Lie subalgebras. Suppose $L=\langle
    U^-,U^+\rangle_{\Lie}$. Let $L_{[p],i}$ be the $p$-envelope of
    $L$ with respect to an injective homomorphism $i:L\lto G$ into a restricted
    Lie algebra $G$. Then $L_{[p],i}=\langle
    U^-_{[p],i},U^+_{[p],i}\rangle_{p-\Lie}$.  
\item Let $L\subset\fgl_V$ be a Lie subalgebra for some finite
    dimensional $k$-vector space $V$ and $A,B,C\in L$ elements. Suppose that $AB=BA=0$ in $\fgl_V$. Then we have $ACB+BCA\in L$. In particular, if $p>2$, then for any $A,B\in L$ with $A^2=0$, we have $(1-A)B(1+A)\in L$. 
\end{enumerate}
\end{lem}

\begin{proof}
(1)  We consider the map $\overline{(~)} \colon L \lto \gr L$ used in Definition \ref{def:FLA} and denote by $\overline S$ the image of $S$ in $\gr L$. We inductively define $S_0=S$,
 $$
 S_{n+1}=\{ [x,y] \ | \ x,y\in S_n\}\cup S_n \subseteq L
 $$
 and $S_\infty=\cup_n S_n$. It is clear that the $k$-vector space generated by $S_\infty$ is the Lie algebra generated by $S$ inside $L$. Similarly, we define $\overline{S}_n\subset\gr L$ for $n\in\N\cup\{\infty\}$ in the same way. 
 
Notice that $\overline{S_0} = \overline S_0$ by definition and, by induction and the definition of the bracket in $\gr L$, we also have $(\overline S)_n \subseteq \overline{S_n}$ for all $n\in \N \cup \{\infty\}$. Thus, if $L'$ is the Lie algebra generated by $S$, then $\overline S_\infty \subseteq \gr L'$ and therefore $\gr L = \gr L'$. Since all Lie algebras considered have finite dimension we can conclude that $L'=L$.

(2) Without loss of generality, we may assume that $L$ is a Lie subalgebra of
$G$ and $i$ is the natural inclusion $L\subset G$. Let
$L'\coloneqq \langle U^-_{[p],i}, U^+_{[p],i}\rangle_{p-\Lie}$. As
$L\subseteq L'\subseteq G$  and $L'$ is a restricted Lie subalgebra of $G$, we
have $L_{[p],i}\subseteq L'$.  However, as $U^-_{[p],i},U^+_{[p],i}\subseteq
L_{[p],i}$, we also have $L'=\langle
U^-_{[p],i},U^+_{[p],i}\rangle_{p-\Lie}\subseteq
L_{[p],i}$. Thus it follows that $L_{[p],i}=L'$. This completes the proof.

(3) Indeed, we have $ACB+BCA=ACB-ABC-CBA+BCA=A[C,B]+[B,C]A=[[B,C],A]\in L$. Let us prove the last assertion. Suppose that  $p>2$ and suppose given $A,B\in L$ with $A^2=0$. Then, we have
\begin{equation*}
(1-A)B(1+A)=B+[A,B]-ABA 
\end{equation*}
in $\fgl_V$. As $p>2$, by the first assertion, we have $ABA=(ABA+ABA)/2\in L$, hence $(1-A)B(1+A)\in L$. This completes the proof. 
\end{proof}

\begin{rem}
For a Lie algebra $L$ over a field $k$, an element $c\in L$ satisfying $\ad(c)\ad(a)\ad(c)=0$ is called a \textit{sandwich element}~(cf.\ \cite[\S1.5]{kos90}). Such elements will play an important role in the present paper. Suppose that $p>2$. Let $c\in L$ be an element such that $\ad(c)^2=0$.  Then for any element $a\in L$, we have $\ad(c)\ad(a)\ad(c)=0$~(cf.\ \cite[\S1.5.3]{kos90}), hence $c$ is a sandwich element. Indeed, for an arbitrary element $b\in L$,  one can compute
\begin{equation*}
\begin{aligned}
\ad(c)\ad(a)\ad(c)(b)
&=\ad(c)(\ad(a)([c,b]))=\ad(c)([\ad(a)(c),b]+[c,\ad(a)(b)])\\
&=\ad(c)([[a,c],b])+\ad(c)^2\ad(a)(b)=\ad(c)([[a,c],b])\\
&=[\ad(c)([a,c]),b]+[[a,c],\ad(c)(b)]=-[\ad(c)^2(a),b]+[[a,c],\ad(c)(b)]\\
&=\ad([a,c])(\ad(c)(b))=(\ad(a)\ad(c)-\ad(c)\ad(a))(\ad(c)(b))\\
&=\ad(a)\ad(c)^2(b)-\ad(c)\ad(a)\ad(c)(b)=-\ad(c)\ad(a)\ad(c)(b),
\end{aligned}
\end{equation*}
which implies that $2\ad(c)\ad(a)\ad(c)(b)=0$. As $p>2$, we get the desired vanishing. 
\end{rem}


\subsection{A criterion for the realization}\label{subsec:criterion}

The following is the key theorem for the proof of the main result.

\begin{thm}\label{thm:AC FLA}
Let $L$ be a Lie algebra over a perfect field $k$ of characteristic $p>0$. Let $U^{+},U^{-}\subseteq L$ be subspaces of $L$ and $\rho:L\hookrightarrow\fgl_V$ an injective homomorphism of Lie algebras for some finite dimensional $k$-vector space $V$. Suppose that the following conditions are satisfied.
\begin{enumerate}
\renewcommand{\labelenumi}{{\rm (\Roman{enumi})}}
\item $\langle U^+,U^-\rangle_{\Lie}=L$.
\item $[U^-,U^-]=0$, and $U^-$ admits a basis $\partial_1,\dots,\partial_m$ such that $\rho(\partial_i)^{l}=0$ for $l\gg 0$. 
\item $U^+$ admits a basis $D_1,\dots,D_n$ such that $\rho(D_i)^2=0$ for any
    $1\le i\le n$. If $p=2$, we further assume that $\rho(D_i)\cdot\rho(D_j)=0$
    for any $i\neq j$, and that $\rho(D_i)\cdot U^-_{[p],\rho}\cdot\rho(D_i)\in
    L_{[p],\rho}$ for any $1\le i\le n$. 
\end{enumerate}
Then $\Gamma\coloneqq
\mathfrak{G}(L_{[p],\rho})\in \pi^\loc_A(\A^1_k)$.
\end{thm}

For the proof, we need the following lemma.

\begin{lem}\label{lem:AC FLA}
Suppose that $k$ is of characteristic $p>2$. Let $V$ be a finite dimensional vector space over $k$. Let $\Gamma\subset\GL_V$ be a subgroup scheme of height one, \emph{i.e.}\ $\Gamma=\Gamma_{(1)}$. Let $\delta:\G_a\lto\GL_V$ be an injective homomorphism of $k$-group schemes with the associated morphism $d\delta:\Lie(\G_{a,k})\lto\fgl_V$ of restricted Lie algebras. Set $D\Def d\delta(1)$. If $D\in \Lie(\Gamma)$ and $D^2=0$ in $\fgl_V$, then the subgroup scheme $\Gamma\subset\GL_V$ is stable under the conjugacy action by $\delta$, \emph{i.e.}\ the map
\begin{equation*}
\G_{a,k}\times \Gamma\lto \GL_{V}\,;\,(t,u)\longmapsto \delta(t)^{-1}u\delta(t) 
\end{equation*}
has image in $\Gamma$.  
\end{lem}

\begin{proof}
As the first Frobenius kernel
$\GL_{V(1)}\Def\Ker(F^{(1)}:\GL_V\lto\GL_V^{(1)})$ is a normal subgroup scheme
of $\GL_{V}$, \emph{i.e.} for any $k$-scheme $T$, the subgroup
$\GL_{V(1)}(T)\subseteq \GL_V(T)$ is normal. Thus given any $T$-point
$\lambda\colon T\to\GL_V$, we get a conjugation map 
$$c_\lambda\colon
T\times\GL_{V}\xrightarrow{~\simeq~}T\times\GL_V$$
$${\hspace{60pt}x\hspace{20pt}}\mapsto {\hspace{20pt}\lambda^{-1}x\lambda\hspace{20pt}}$$
which preserves the subgroup $\GL_{V(1)}\times T$. We will apply this to
$T=\G_{a,k}$ and 
$\lambda = \delta$.

As $\Gamma$ is of height one, we have $\Gamma\subseteq\GL_{V(1)}$. We have to
show that the composition
\begin{equation*}
    \xi:\G_{a,k}\times \Gamma\subseteq
    \G_{a,k}\times\GL_{V(1)}\xrightarrow{~c_{\delta}~}\G_{a,k}\times\GL_{V(1)}\xrightarrow{\textup{pr}_2}\GL_{V(1)} 
\end{equation*}
has image in $\Gamma$. Let $K=k(t)$ be the function field of $\G_{a,k}$. It is enough to show that $\xi(\Spec K\times_k\Gamma)\subseteq\Gamma$, \emph{i.e.}\ the induced group homomorphism
\begin{equation*}
\psi:\Gamma_K\lto\GL_{V(1),K}
\end{equation*} 
factors through in $\Gamma_K$. However, as $\Gamma_K$ and $\GL_{V(1),K}$ are of height one, it suffices to show that the associated morphism of restricted Lie algebras
\begin{equation*}
d\psi:\Lie(\Gamma)_K\lto\fgl_{V,K}
\end{equation*} 
factors through $\Lie(\Gamma)_K$. Notice that $d\psi$ is the composition of the natural inclusion $\Lie(\Gamma)_K\hookrightarrow\fgl_{V,K}$ together with the differential $dc:\fgl_{V,K}\xrightarrow{~\simeq~}\fgl_{V,K}$ of the conjugacy action by $\delta$, which is just
\begin{equation*}
A\longmapsto \delta(t)^{-1}A\delta(t). 
\end{equation*}
However, by our assumption, we have $\delta(t)=\exp(t D)=1+tD$, and hence $\delta(t)^{-1}=\delta(-t)=1-tD$. Thus, by Lemma \ref{lem:GLA vs FLA}(3), we have $\delta(t)^{-1}A\delta(t)\in\Lie(\Gamma)_K$ for any $A\in\Lie(\Gamma)_K$. This completes the proof. 
\end{proof}

\begin{proof}[Proof of Theorem \ref{thm:AC FLA}]
Without loss of generality, we may assume that $L$ is a Lie subalgebra of $\fgl_V$ and $\rho$ is the natural inclusion $L\subset\fgl_V$. 
By the condition (III), $U^+$ admits a basis $D_1,\dots,D_{n}$ such that $D_i^2=0$ for any $1\le i\le n$. Thanks to Theorem \ref{thm:f_X}, for any $1\le i\le n$, we get an injective homomorphism of group schemes
\begin{equation*}
\delta_i\Def f_{D_{i}}:W_1=\G_{a,k}\lto\GL_V
\end{equation*}
(cf.\ (\ref{eq:f_X})) which recovers the natural inclusion of Lie algebras, \emph{i.e.}\
\begin{equation*}
d\delta_i=df_{D_{i}}:k\simeq k\cdot D_i\lto\fgl_V~;~x\longmapsto x D_i,
\end{equation*}
(cf.\ (\ref{eq:df_X})). We define the morphism $\delta:\G_{a,k}^n\lto\GL_V$ of
$k$-schemes to be the product of the exponential maps $\delta_i~(1\le i\le n)$, \emph{i.e.}
\begin{equation*}
\delta:\G_{a,k}^n\lto\GL_V\,;\,(t_1,\dots,t_n)\longmapsto\delta_1(t_1)\cdots\delta_n(t_n). 
\end{equation*}

On the other hand, by the condition (II), $U^-$ admits a basis $\partial_1,\dots,\partial_{m}$ satisfying $\partial_1^{l}=\cdots=\partial_m^{l}=0$ for $l\gg 0$. For each $1\le i\le m$, let $n_i$ denote the positive integer satisfying $\partial_{i}^{p^{n_i}}=0$ and $\partial_i^{p^{n_i-1}}\neq 0$. 
Then, by Theorem \ref{thm:f_X}, we get a homomorphism~(cf.\ Remark \ref{rem:f_X})
\begin{equation*}
\omega\Def f_{\partial_1}\cdots f_{\partial_{m}}\colon W_{\bfn}\Def W_{n_1}\times\cdots \times W_{n_m}\lto\GL_V
\end{equation*} 
inducing the natural inclusions of Lie algebras
\[
d\omega=\sum_{i=1}^m df_{\partial_i}:
k^{n_1+\cdots+n_m}\simeq U^-_{[p],\rho}\lto \fgl_V~;~(x^i_{0},\dots,x^i_{n_i-1})_{i=1}^m\longmapsto\sum_{i=1}^m\sum_{j=0}^{n_i-1}x^i_j\partial_{i}^{p^j}
\]
(cf.\ (\ref{eq:df_X})).

Let us consider the morphism of affine $k$-schemes
\begin{equation*}
\phi:W_{\bfn}\times\G_{a,k}^{n}\lto\GL_{V}, \ \ \phi(t,s)=\omega(t)\delta(s)
\end{equation*}
Let $\Gamma$ be the finite group scheme of height one associated with the restricted Lie algebra $L_{[p],\rho}$. 
To the natural inclusion $L_{[p],\rho}\hookrightarrow\fgl_{V}$, one can uniquely associate the injective homomorphism of finite group schemes over $k$,
\begin{equation*}
\Gamma\hookrightarrow\GL_{V(1)}.
\end{equation*}
In the following, we consider $\Gamma$ as a subgroup scheme by this injective homomorphism.
By our construction, the restrictions
\[
\omega\colon W_{\bfn(1)} \lto \GL_{V(1)},\ \ \delta_i\colon \alpha_{p,k} \lto \GL_{V(1)}\,(1\le i\le n)
\]
induce the natural inclusions $U^-_{[p],\rho} \lto \fgl_V$ and $k\cdot D_i \lto
\fgl_V$ on Lie algebras and thus factors through $L_{[p],\rho}$. It follows
that $\im(\omega),\im(\delta_1),\dots,\im(\delta_n)\subseteq \Gamma$. 
By condition (I) together with Lemma \ref{lem:GLA vs FLA} (2), we can conclude that

\begin{equation}
    L_{[p],\rho} = \langle U_{[p],\rho}^-, k\cdot D_1,\cdots,
    k\cdot D_n\rangle_{p-\Lie} \subseteq \fgl_V
\end{equation}

By \cite[Chapter II, \S 7, n$^\textup{o}$4, 4.3]{dg70}, the functor $\Lie(-)$
induces an order preserving bijection between the restricted Lie subalgebras
of $\fgl_V$ and the closed subgroup schemes of
$\GL_{V(1)}$. Thus the smallest restricted Lie subalgebra containing
$U_{[p],\rho}^-, k\cdot D_1,\cdots,
    k\cdot D_n$ corresponds to the smallest closed subgroup
    scheme containing
    $\im(\omega),\im(\delta_1),\dots,\im(\delta_n)$, namely we have 

\begin{equation}\label{eq:Gamma generators}
\Gamma=\langle
\im(\omega),\im(\delta_1),\dots,\im(\delta_n)\rangle\subseteq\GL_{V(1)}.
\end{equation} 

Now consider the $\GL_{V(1)}$-torsor $P\lto W_{\bfn}^{(1)}\times\G^{n(1)}_{a,k}$ given by the Cartesian diagram
\begin{equation*}
\begin{xy}
\xymatrix{\ar@{}[rd]|{\square}
P\ar[r]\ar[d]&\GL_{V}\ar[d]^{F^{(1)}}\\
W_{\bfn}^{(1)}\times\G^{n(1)}_{a,k}\ar[r]_{~~~~\phi^{(1)}}&\GL_{V}^{(1)}.
}
\end{xy}
\end{equation*}
Let $G\subset\GL_{V(1)}$ be the image of the associated  homomorphism $\pi^{\loc}\bigl(W_{\bfn}^{(1)}\times\G^{n(1)}_{a,k}\bigl)\lto \GL_{V(1)}$. We are going to show that $\Gamma=G$, which, by Lemma \ref{lem:bertini},  completes the proof of the proposition. 

We start by showing that  $\Gamma\subseteq G$. By (\ref{eq:Gamma generators}), it is enough to show that 
\begin{equation*}
    \im(\omega),\im(\delta_1),\dots,\im(\delta_n)\subseteq G.
\end{equation*} 
For that, it suffices to show that for any factor $\iota\colon
W_{n_i}^{(1)} \lto \GL_V^{(1)}$ (resp. 
$\iota\colon W_1^{(1)}=\G_{a,k}^{(1)}\to \GL_V^{(1)}$) the image of
$\pi^{\loc}(W_{n_i}^{(1)})\lto \GL_{V(1)}$ is $W_{n_i(1)}$ (resp. the image of $\pi^{\loc}(W_{1}^{(1)})=\pi^{\loc}(\G_{a,k}^{(1)})\lto \GL_{V(1)}$ is
$W_{1(1)}={\G_{a,k}}_{(1)}$).  The map $\iota$ is a group homomorphism and induces an injective map on Lie algebras. Let $l$ denote 1 or $n_i$ for $1\leq
i\leq m$.  Then there is a commutative diagram 
\begin{equation*}
\begin{xy}
\xymatrix{
W_{l}\ar[r]^{~~\iota~~~}\ar[d]_{F^{(1)}}&\GL_V\ar[d]^{F^{(1)}}\\
W_{l}^{(1)}\ar[r]_{\iota^{(1)}~}&\GL_{V}^{(1)},
}
\end{xy}
\end{equation*}
where the upper horizontal arrow is $W_{l(1)}$-equivariant. 
The $W_{l(1)}$-torsor $W_l\to
W_l^{(1)}$ induces a factorization $\pi^{\loc}(W_l^{(1)})\xrightarrow \alpha W_{l(1)}\subseteq \GL_{V(1)}$. If $H$ is the image of $\alpha$ then there exists an $H$-torsor $P\to W_l^{(1)}$ with an $H$-equivariant map $P\to W_l$. This map is a surjective closed immersion because $W_{l(1)}$ is local, hence an isomorphism because $W_l$ is reduced. Thus $\alpha$ is surjective and therefore $\Gamma\subseteq G$.


Let us prove now that $G\subseteq \Gamma$. It suffices to show that the morphism $\phi^{(1)}:W^{(1)}_{\bfn}\times\G_{a,k}^{n(1)}\lto\GL_{V}^{(1)}$ lifts to the middle quotient $\GL_{V}/\Gamma$, 
\begin{equation*}
\begin{xy}
\xymatrix{
&\GL_{V}\ar[d]\ar@/^14mm/[dd]^{F^{(1)}}\\
&\GL_{V}/\Gamma\ar[d]\\
W^{(1)}_{\bfn}\times\G_{a,k}^{n(1)}\ar[r]_{~~~~~\phi^{(1)}}\ar@{-->}[ru]&\GL_{V}^{(1)}.
}
\end{xy}
\end{equation*}
Note that the following diagram is commutative, 
\begin{equation*}
\begin{xy}
\xymatrix{
W_{\bfn}\times\G_{a,k}^{n}\ar[r]^{~~\phi}\ar[d]_{F^{(1)}}&\GL_V\ar[d]^{F^{(1)}}\\
W_{\bfn}^{(1)}\times\G_{a,k}^{n(1)}\ar[r]_{~~~~\phi^{(1)}}&\GL_{V}^{(1)}.
}
\end{xy}
\end{equation*}
Therefore, to get a desired lift of $\phi^{(1)}$, it suffices to check that the composition 
\begin{equation*}
W_{\bfn}\times\G_{a,k}^{n}\xrightarrow{~\phi~}\GL_V\lto\GL_{V}/\Gamma
\end{equation*}
factors through the quotient $W_{\bfn}^{(1)}\times\G_{a,k}^{n(1)}=(W_{\bfn}\times\G_{a,k}^{n})/(W_{\bfn(1)}\times\alpha_{p,k}^{n})$. 

For $(s,t)\in W_{\bfn}\times\G^{n}_{a,k}$ and $(u,v)\in W_{\bfn(1)}\times\alpha_p^{n}$, we have
\[
\phi(s+u,t+v)=\omega(s)\omega(u)\delta(t+v)=\phi(s,t)\zeta(t,u,v)
\]
with
\[
\zeta(t,u,v)=\delta(t)^{-1}\omega(u)\delta(t+v)\in \GL_{V(1)}
\]
Here we are using the functorial point of view, so that $s,u,t,v$ should be
thought of as sections of the corresponding functor over some fixed scheme.
\todo[inline]{Was this the problem of the referee?}

We need to prove that $\zeta(t,u,v)\in \Gamma$. First let us handle the case when $p>2$, in which case we have
\begin{equation*}
\begin{aligned}
\zeta(t,u,v)
=&\bigl(\delta(t)^{-1}\omega(u)\delta(t)\bigl) \cdot\bigl(\delta(t)^{-1}\delta(t+v)\bigl)\\
=&\bigl(\delta_n(t_n)^{-1}\cdots\delta_1(t_1)^{-1}\omega(u)\delta_1(t_1)\cdots\delta_n(t_n)\bigl)\cdot\\
&~\prod_{i=1}^n\bigl(\delta_n(t_n)^{-1}\cdots\delta_{i+1}(t_{i+1})^{-1}\delta_i(v_i)\delta_{i+1}(t_{i+1})\cdots\delta_n(t_n)\bigl).
\end{aligned}
\end{equation*} 
Hence, by applying Lemma \ref{lem:AC FLA} iteratively, one can conclude that $\zeta(t,u,v)\in\Gamma$. This completes the proof in the case where $p>2$.

Let us assume that $p=2$. Then by the condition (III), $\delta_i$ and $\delta_j$ commute with each other. Hence, we have
\begin{equation*}
\zeta(t,u,v)=\delta(t)^{-1}\omega(u)\delta(t)\delta(v).
\end{equation*}
Since $\delta(\alpha_{p,k}^n)\subseteq \Gamma$, it is enough to show that the map
\[
\xi\colon \G_{a,k}^n\times W_{\bfn(1)} \lto \GL_{V(1)}, \ \ \xi(t,u)=\delta(t)^{-1}\omega(u)\delta(t)
\]
has image in $\Gamma$. Let $K$ be the fraction field of $\G_{a,k}^n$. It is enough to show that  $\xi(\Spec K \times_k W_{\bfn(1)})\subseteq \Gamma$ and, therefore, that the group homomorphism
\[
\psi\colon W_{\bfn(1)} \lto\GL_{V(1),K}
\]
factors through $\Gamma_K$. However, as it is a homomorphism of finite group schemes of height one, it suffices to show that the corresponding morphism
\begin{equation*}
d\psi :\Lie(W_{\bfn})\otimes_k K\lto\fgl_{V}\otimes_k K
\end{equation*} 
of the restricted Lie algebras factors through $L_{[p],\rho}\otimes K$.
By construction the composition
\[
\Spec K \lto \G_{a,k}^n \lto \GL_V
\]
is given by $\delta(B)=\prod_{i=1}^n\exp(b_iD_i)\in \GL_V(K)$ for some $B=\sum_{i=1}^nb_i D_i\in L_h\otimes_kK $. 
Since the map $\psi\colon W_{\bfn(1)} \lto\GL_{V(1),K}$ is obtained as composition of $\omega\colon W_{\bfn(1)} \lto\GL_{V(1),K}$ and the conjugation by $\delta(B)$, $c\colon \GL_{V(1),K} \lto\GL_{V(1),K}$ the corresponding map on Lie algebras in the composition of the inclusion $d\omega\otimes K\colon \Lie(W_{\bfn})\otimes K=U^-_{[p],\rho}\otimes K\lto \fgl_V \otimes K$ and the differential $\fgl_V \otimes K\lto \fgl_V \otimes K$ of $c$, which is just
\[
A\longmapsto \delta(B)^{-1}A\delta(B).
\]
As we have chosen the basis $D_1,\dots,D_n$ so that $D_i^2=0$, we have 
\begin{equation*}
\delta(B)=\prod_{i=1}^n(1+b_iD_i). 
\end{equation*}
Again by the condition in (III), we have $B^2=0$, hence 
\begin{equation*}
 \delta(B)=\exp(B)=1+B. 
\end{equation*} 
Thus, for $\nabla\in U^-$, we have to show that
\[
 \delta(B)^{-1}\nabla^{p^j}\delta(B) =(1-B)\nabla^{p^j}(1+B)\in L_{[p],\rho}\otimes K.
\]
Therefore, by the condition (III) together with the first assertion in Lemma \ref{lem:GLA vs FLA}(3), we get
\begin{equation*}
\begin{aligned}
 \delta(B)^{-1}\nabla^{p^j}\delta(B)
&= (I-B)\nabla^{p^j}(I+B)\\
&=\nabla^{p^j} - [B,\nabla^{p^j}] + B\nabla^{p^j} B\\
&=\nabla^{p^j}-[B,\nabla^{p^j}]+\sum_{i=1}^nb_i^2 D_i\nabla^{p^j}D_i+\sum_{i<i'}b_ib_{i'}(D_i\nabla^{p^j}D_{i'}+D_{i'}\nabla^{p^j}D_i)\in L_{[p],\rho}\otimes K,
\end{aligned}
\end{equation*}
so that the claim follows. This completes the proof.    
\end{proof}

For $\Z$-graded simple Lie algebras, we get the following consequence.

\begin{cor}\label{cor:AC GLA}
Let $L$ be a center-free $\Z$-graded Lie algebra over a perfect field $k$ of characteristic $p>0$ whose grading is given by 
\begin{equation*}
L=\bigoplus _{i=-s}^{h}L_i~~\text{with}~~s=s(L),~~h= h(L), 
\end{equation*}
and let $U\subseteq L_{\le -1}$ be a subspace. Let $L_{[p]}$ be the $p$-envelope of $L$ in the adjoint representation $\ad:L\hookrightarrow\Der(L)$. 
Suppose that the following conditions are satisfied.
\begin{enumerate}
\renewcommand{\labelenumi}{{\rm (G\,\Roman{enumi})}}
\item $[U,U]=0$.
\item $L=\langle U,L_{h}\rangle_{\Lie}$. 
\item If $p>2$, then we assume $s<h$. If $p=2$, then we further assume there exists an integer $l>0$ such that $\ad(U)^{2^{l}}=0$ and $(2^{l-1}+1)s<h$.
\end{enumerate}
Then $\Gamma\coloneqq
\mathfrak{G}(L_{[p]})\in \pi^\loc_A(\A^1_k)$. 
\end{cor}

\begin{proof}
By condition (G\,I), $U$ is an abelian Lie subalgebra of $L$. The condition
(G\,III) particularly says $h>0$, and hence $[L_{h},L_h]\subseteq L_{2h}=0$.
Therefore, $L_h$ is also an abelian Lie subalgebra. If $s=0$, then $U=0$, in
which case the condition (G\,II) is not satisfied. Thus we have $s>0$. If one
takes $l>0$ so that $h+s<l$, then we have
\begin{equation*}
\ad(U)^{l}(L)\subseteq\bigoplus_{i=-s}^h\bigoplus_{k=-s}^{-1}L_{i+kl}=0,
\end{equation*}
hence the conditions (I), (II) in Theorem \ref{thm:AC FLA} are satisfied for
$U^+\Def L_h$, $U^-\Def U$ with respect to the adjoint representation
$\ad:L\lto\Der(L)\subset\fgl_L$. Therefore, to prove the corollary, it suffices
to show that the condition (III) in Theorem \ref{thm:AC FLA} is fulfilled as
well. By the condition (G\,III), we have $h>s$. Then we have
\begin{equation*}
\ad(L_h)^2(L)\subseteq\bigoplus_{i=-s}^hL_{h+(h+i)}=0.
\end{equation*}
Hence the condition (III) is fulfilled in the case when $p>2$. Let us assume that $p=2$. Again by the condition (G\,III), for any $0\le j\le l-1$, we have
\begin{equation*}
\begin{aligned}
\ad(L_h)\ad(U)^{2^j}\ad(L_h)(L)
&\subseteq\ad(L_h)\ad(U)^{2^j}\Bigl(\bigoplus_{-s\le i\le 0}L_{h+i}\Bigl)\\
&\subseteq\bigoplus_{-s\le i\le 0}\ad(L_h)\Bigl(\bigoplus_{-s\le k\le -1} L_{h+i+2^{j}k}\Bigl)\\
&\subseteq\bigoplus_{-s\le i\le 0}\bigoplus_{-s\le k\le -1}L_{h+(h+i+2^jk)}=0.
\end{aligned}
\end{equation*}
As $U^-_{[p]}=\sum_{0\le j\le l-1}U^{p^j}$, this implies that the condition (III) in Theorem \ref{thm:AC FLA} is also satisfied. Thus, the assertion is a consequence of Theorem \ref{thm:AC FLA}. This completes the proof. 
\end{proof}

\begin{rem}\label{rem:AC Cartan}
Let $\Gamma\subset\GL_{V}$ be a finite local subgroup scheme of height one of
the general linear algebraic group associated with a finite dimensional
$k$-vector space $V$. Suppose that $\Gamma$ can be generated by finite simple
unipotent subgroup schemes $H_i\simeq\alpha_p~(i\in I)$, \emph{i.e.}\
$\Gamma=\langle H_i\,|\, i\in I\rangle$. Then, as discussed in the proof of
Proposition \ref{thm:AC FLA}, if $\delta_i:\G_a\lto\GL_V~(i\in I)$ denote the exponential maps associated with the homomorphisms $\fg_a\simeq\Lie(H_i)\subset\Lie(\Gamma)\lto\fgl_V$ of restricted Lie algebras,  then the pullback of the relative Frobenius $F^{(1)}:\GL_V\lto\GL_V^{(1)}$ along the morphism
\begin{equation*}
\A_k^{I}=\prod_{i\in I}\A^1_k\xrightarrow{~\delta^{(1)}=\prod_{i\in I}\delta_i^{(1)}}\GL_V^{(1)}
\end{equation*}
defines a $\GL_{V(1)}$-torsor $P\lto \A^I_k$. One might expect that the resulting torsor $P\lto\A^I_k$ would be always reduced to a $\Gamma$-torsor. However, this is not necessarily true~(cf.~Example \ref{ex:Rumynin}). Therefore,  one really needs to take care about the choice of generators of $\Gamma$ to control the image of the associated homomorphism $\pi^{\loc}(\A^I_k)\lto\GL_{V(1)}$. This problem stems from the failure of the formula 
\begin{equation*}
\exp(B) A\exp(-B)=\exp(\ad B)A
\end{equation*}
in positive characteristic~(cf.\ \cite[\S4]{mat05}). See also  Example \ref{ex:counter-example in char p=3}, in which we will see that the condition $D^2=0$ in Lemma \ref{lem:AC FLA} cannot be removed.  
\end{rem}

\begin{ex}(Proposed by Dmitriy Rumynin)\label{ex:Rumynin}
Let $k$ be an algebraically closed field of characteristic $p=2$. Let $L_{\C}\Def\fsl_2(\C)$ with 
\begin{equation*}
e=
\begin{pmatrix}
0&1\\
0&0
\end{pmatrix},
f=
\begin{pmatrix}
0&0\\
1&0
\end{pmatrix},
h=
\begin{pmatrix}
1&0\\
0&-1
\end{pmatrix} 
\end{equation*}
the standard basis $L_{\C}=\C e\oplus \C h\oplus \C f$. Note that
\begin{equation*}
[e,f]=h,[h,e]=2e,[h,f]=-2f.
\end{equation*}
Then by taking reductions modulo $p=2$, we get the following three non-isomorphic restricted Lie algebras of dimension $3$. 
\begin{enumerate}
\renewcommand{\labelenumi}{(\arabic{enumi})}
\item $\fsl_{2,k}\simeq(\Z e\oplus\Z f\oplus\Z h)\otimes_{\Z} k$.
\item $\mathfrak{pgl}_{2,k}\simeq(\Z e\oplus\Z f\oplus\Z\frac{1}{2} h)\otimes_{\Z} k$.
\item $L\Def (\Z e\oplus \Z\frac{1}{2} f\oplus
    \Z\frac{1}{2} h)\otimes_{\Z} k$.
\end{enumerate}
The first two restricted Lie algebras are not simple. The last one $L$ is simple. Set
\begin{equation*}
\tilde{e}\Def e\otimes 1,~\tilde{f}\Def \frac{1}{2}f\otimes 1,~\tilde{h}\Def \frac{1}{2}h\otimes 1.
\end{equation*}
Then $L=k\tilde{e}\oplus k\tilde{f}\oplus k\tilde{h}$ with $\tilde{e}^{[2]}=\tilde{f}^{[2]}=0$ and $\tilde{h}^{[2]}=\widetilde{h}$. Moreover, we have
\begin{equation*}
[\tilde{e},\tilde{f}]=\tilde{h},[\tilde{h},\tilde{e}]=\tilde{e},[\tilde{h},\tilde{f}]=\tilde{f}.
\end{equation*}
Hence, the $2$-nilpotent elements $\tilde{e}$ and $\tilde{f}$ generate the Lie algebra $L$. 

Set $\Gamma\coloneqq \mathfrak{G}(L)$. Then the adjoint representation
\begin{equation*}
\ad\colon L\lto\fgl_L 
\end{equation*} 
corresponds to a homomorphism $\Gamma\subset \GL_L$. Let
$\delta_i\colon \G_{a,k}\lto \GL_L$ ($i=1,2$) be the two
homomorphisms corresponding to the inclusions $k\,\tilde{e} \xrightarrow{\ad}\fgl_L$
and $k\,\tilde{f} \xrightarrow{\ad} \fgl_L$ respectively. Consider the map
$\lambda\colon\G_{a,k}\to \GL_L$ associated with the element 
\[
   (1+\ad(\tilde{f}))\ad(\tilde{e})(1+\ad(\tilde{f}))\in \fgl_L
\]
For any $k$-scheme $T$ and any $u\in \G_{a,k}(T)$, we have
$$\lambda(u)=(1+\ad(\tilde{f}))(1+\ad(\tilde{e})u)(1+\ad(\tilde{f}))=\delta_2(1)\delta_1(u)\delta_2(1)$$

If the $\GL_{V(1)}$-torsor $P\lto \A_k^2$, which is defined as
the pullback of the relative Frobenius $F^{(1)}:\GL_V\lto\GL_V^{(1)}$ along the  map 
$$\G_{a,k}\times\G_{a,k}\xrightarrow{~\delta_1\cdot\delta_2~}
\GL_V^{(1)},$$
could be reduced to a $\Gamma$-torsor, then just as was shown in
Theorem \ref{thm:AC FLA}, one should have
$$\delta_1(s+u)\delta_2(t+v) = \delta_1(s)\delta_2(t)\cdot
\gamma$$
for all $s,t\in\G_{a,k}(T), u,v\in\alpha_{p,k}(T)$ and for some
$\gamma\in\Gamma(T)$ (depending on $s,t,u,v$). This is true if
and only if $\delta_2(-t)\delta_1(u)\delta_2(t+v)\in\Gamma(T)$,
or equivalently $\delta_2(t)\delta_1(u)\delta_2(t)\in\Gamma(T)$
(note that $\delta_2(t)=\delta_2(-t)$ and $\delta_2(v)\in
\Gamma(T)$). Thus the restriction of $\lambda$ to the subgroup
$\alpha_{p,k}\subset \G_{a,k}$ takes values in
$\Gamma\subset\GL_L$. Now considering the map $\Lie(\lambda)$ we get
\[
    (1+\ad(\tilde{f}))\ad(\tilde{e})(1+\ad(\tilde{f}))\in
    \ad(L)\subset \fgl_L
\]
\todo[inline]{Check the argument above.}
However, this
is not the case. Indeed, one can check that
\begin{equation*}
(1+\ad(\tilde{f}))\ad(\tilde{e})(1+\ad(\tilde{f}))=\ad(\tilde{e})+\ad(\tilde{h})+\ad(\tilde{f})\ad(\tilde{h})\not\in \ad(L), 
\end{equation*}
or equivalently
\begin{equation*}
\ad(\tilde{f})\ad(\tilde{h})\not\in\ad(L). 
\end{equation*}
Indeed, if there exist $c_1,c_2,c_3\in k$ such that
\begin{equation*}
\ad(\tilde{f})\ad(\tilde{h})=c_1\ad(\tilde{e})+c_2\ad(\tilde{f})+c_3\ad(\tilde{h}),  
\end{equation*}
then we have
\begin{equation*}
0=\ad(\tilde{f})\ad(\tilde{h})(\tilde{f})=c_1\ad(\tilde{e})(\tilde{f})+c_2\ad(\tilde{f})(\tilde{f})+c_3\ad(\tilde{h})(\tilde{f})=c_1\tilde{h}+c_3\tilde{f},
\end{equation*}
hence $c_1=c_3=0$. Moreover, we have $0=\ad(\tilde{f})\ad(\tilde{h})(\tilde{h})=c_2\ad(\tilde{f})(\tilde{h})=c_2\tilde{f}$, 
hence $c_2=0$. However, this contradicts the fact that
\begin{equation*}
\ad(\tilde{f})\ad(\tilde{h})(\tilde{e})=\tilde{h}\neq 0.
\end{equation*}
Therefore, we conclude that $\ad(\tilde{f})\ad(\tilde{h})\not\in\ad(L)$. 
\end{ex}


\section{The non-classical simple Lie algebras in characteristic $p>3$}\label{sec:Cartan}

Let $\N$ denote the set of non-negative integers, \emph{i.e.}\ $\N=\Z_{\ge
0}=\{n\in\Z~|~n\ge 0\}$, which we consider as a commutative ordered monoid in
the usual way. For any positive integer $m>0$, $\N^m$ denotes the set of
$m$-tuples of non-negative integers, which we consider as a commutative
partially ordered monoid by the component-wise addition and by the rule that
for any $m$-tuples $\alpha,\beta\in\N^{m}$, the inequality $\alpha\le\beta$
holds if and only if $\alpha_i\le\beta_i$ for each $1\le i\le m$. For any element $\alpha=(\alpha_1,\dots,\alpha_m)\in\N^m$, we set
\begin{equation}\label{eq:|alpha|}
|\alpha|\Def\alpha_1+\cdots+\alpha_m\in\N. 
\end{equation}
For each $1\le i\le m$, we define an element $\epsilon_i\in\N^m$ to be 
\begin{equation}\label{eq:epsilon_i}
\epsilon_i=(0,\dots,0,\overbrace{1}^{i},0,\dots,0),
\end{equation}
where $1$ lies only in the $i$-th component. For any two elements $\alpha,\beta\in\N^m$, we set
\begin{equation*}
\binom{\alpha}{\beta}\Def\prod_{i=1}^m\binom{\alpha_i}{\beta_i}.
\end{equation*}
Consider the polynomial ring $\Q[X_1,\dots,X_m]$. For each $\alpha\in\N^m$, we set
\begin{equation*}
X^{(\alpha)}\Def\prod_{i=1}^mX_i^{(\alpha_i)}~~\text{with $X_i^{(\alpha_i)}=\frac{X_i^{\alpha_i}}{\alpha_i!}$}
\end{equation*}
in $\Q[X_1,\dots,X_m]$. Note that for any elements $\alpha,\beta\in\N^m$, we have
\begin{equation*}
X^{(\alpha)}X^{(\beta)}=\binom{\alpha+\beta}{\alpha}X^{(\alpha+\beta)}. 
\end{equation*}
Thus the $\Z$-span $\sum_{\alpha\in\N^m}\Z X^{(\alpha)}$ of $\{X^{(\alpha)}~|~\alpha\in\N^m\}$ in $\Q[X_1,\dots,X_m]$ gives a commutative $\Z$-algebra, which we denote by
\begin{equation}\label{eq:PD poly}
A(m)_{\Z}\Def\sum_{\alpha\in\N^m}\Z X^{(\alpha)}\subset\Q[X_1,\dots,X_m]. 
\end{equation}

\subsection{The Witt algebras}\label{subsec:witt}

In this subsection, we recall the definition of the \textit{Witt algebras} following \cite[\S4.2]{Strade}, which provide us with typical examples of non-classical type simple Lie algebras.  
Let $k$ be a field of characteristic $p>0$. For any integer $m\ge 1$, we define the $k$-algebra $A(m)$ to be
\begin{equation*}
A(m)\Def k\otimes_{\Z}A(m)_{\Z}
\end{equation*}
(cf.\ (\ref{eq:PD poly})). By abuse of notation, we still write $X^{(\alpha)}\in A(m)$ for the image of $X^{(\alpha)}\in A(m)_{\Z}$ in $A(m)$. 

For any $m$-tuple $\bfn\in\Z_{>0}^m$ of positive integers, we define the finite dimensional $k$-algebra $A(m;\bfn)$ to be the subalgebra of $A(m)$,
\begin{equation*}
A(m;\bfn)\Def \sum_{0\le\alpha\le\bfp^{\bfn}-\unit}k\cdot X^{(\alpha)},
\end{equation*}
where $\bfp^{\bfn}-\unit=(p^{n_1}-1,\dots,p^{n_m}-1)$. 
Then the algebra $A(m;\bfn)$ has a natural grading of $k$-vector subspaces
\begin{equation}\label{eq:gr A(m;n)}
A(m;\bfn)=\bigoplus_{i\ge 0}A(m;\bfn)_i\quad\text{with}\quad A(m;\bfn)_i=\sum_{|\alpha|=i}k\cdot X^{(\alpha)}.
\end{equation}
By definition, 
\begin{equation}\label{eq:tau(n)}
A(m;\bfn)_i=0\quad\text{for}\quad i>|\bfp^{\bfn}-\unit|=\sum_{i=1}^mp^{n_i}-m.
\end{equation}
and $A(m;\bfn)_{|\bfp^{\bfn}-\unit|}=k\cdot X^{(\bfp^{\bfn}-\unit)}$ is of dimension one

Now we define the \textit{$m$-th Witt algebra $W(m;\bfn)$ of weight $\bfn$} as the Lie subalgebra 
\begin{equation}\label{eq:def W}
W(m;\bfn)\Def\sum_{i=1}^mA(m;\bfn)\partial_i\subseteq\Der_k(A(m;\bfn)), 
\end{equation}
where the derivations $\partial_i~(1\le i\le m)$ are defined by
\begin{equation*}
\partial_i(X^{(\alpha)})=X^{(\alpha-\epsilon_i)},
\end{equation*} 
which give a  basis of $W(m;\bfn)$ over $A(m;\bfn)$, \emph{i.e.}
\begin{equation*}
W(m;\bfn)=\bigoplus_{i=1}^m A(m;\bfn)\partial_i.
\end{equation*}
The grading (\ref{eq:gr A(m;n)}) of $A(m;\bfn)$ gives the $\Z$-graded structure on the Lie algebra $W(m;\bfn)$~(cf.\ Definition \ref{def:GRLA}) as follows,
\begin{equation}\label{eq:witt gr}
W(m;\bfn)=\bigoplus_{i=-1}^{|\bfp^{\bfn}-\unit|-1}W(m;\bfn)_i\quad\text{with}\quad W(m;\bfn)_i=\sum_{j=1}^m A(m;\bfn)_{i+1}\partial_j.
\end{equation}
With respect to the gradation, the Witt algebra $W(m;\bfn)$ has depth $s(W(m;\bfn))=1$ and height $h(W(m;\bfn))=|\bfp^{\bfn}-\unit|-1$. The following formula are then easy to verify~(cf.\ \cite[\S4.2]{Strade}). 
\begin{itemize}
\item For any $\alpha\in\N^m$ and any $1\le i,j\le m$, \begin{equation}\label{eq:witt bracket}
[\partial_i,X^{(\alpha)}\partial_j]=X^{(\alpha-\epsilon_i)}\partial_j.
\end{equation}
\item For any $0\le i\le h(W(m;\bfn))$, we have $[W(m;\bfn)_{-1},W(m;\bfn)_i]=W(m;\bfn)_{i-1}$.
\item For any $1\le i\le m$, we have
\begin{equation}\label{eq:partial p^n=0}
\ad(\partial_i)^{p^{n_i}}=0\quad(1\le i\le m)
\end{equation}
in $\Der(W(m;\bfn))$. 
\end{itemize}
Indeed, the second and third equations are immediate from the first one (\ref{eq:witt bracket}). Note that the second one  especially says that the Witt algebra $W(m;\bfn)$ is generated by the subspaces $W(m;\bfn)_{-1}$ and $W(m;\bfn)_{h(W(m;\bfn))}$.  

\begin{lem}\label{lem: W(m,n)} If $k$ is of characteristic $p>3$ or $(p=3, m>1)$, or
    $(p=3, \bfn>\unit)$, and if we set
    $U\coloneqq W(m;\bfn)_{-1}$, then the Witt
    algebra $W(m;\bfn)$ with its grading satisfies the conditions {\rm(G\, I), (G\,
    II), (G\, III)} in Corollary \ref{cor:AC GLA}.
\end{lem}
\begin{proof} Let us check the conditions one by one as follows:
    \begin{itemize} 
        \item (G\, I) follows from the fact that
            $s(W(m;\bfn))=1$;
        \item (G\, II) follows from the formula:
              $[W(m;\bfn)_{-1},W(m;\bfn)_i]=W(m;\bfn)_{i-1};$
        \item (G\, III) follows from the fact that
            $h(W(m;\bfn))=|\bfp^{\bfn}-\unit|-1\geq 2$ when
            $p>3$ or $(p=3, m>1)$, or
    $(p=3, \bfn>\unit)$.
    \end{itemize}
\end{proof}
\todo[inline]{Is this OK?}

\begin{rem}\label{rem:Witt alg}
It is known that the Witt algebra $W(m;\bfn)$ is simple unless $(m,\bfn,p)=(1,1,2)$. In characteristic $p=3$, then the Witt algebra $W(1;1)$ is isomorphic to $\fsl_2$ as a Lie algebra~(cf.\ \cite[\S4.4 (B)]{Strade}), hence it is of classical type~(cf.\ Remark \ref{rem:classical p=3}). 
\end{rem}

\begin{rem}
If $\bfn=\unit$, then $A(m;\unit)\simeq k[T_1,\dots,T_m]/(T_1^p,\dots,T_m^p)$ as a $k$-algebra. Moreover it turns out that
\begin{equation*}
W(m;\unit)=\sum_{i=1}^mA(m;\unit)\partial_i=\Der_k(A(m;\unit)), 
\end{equation*}
hence it has a restricted structure given by the usual $p$-th multiplication of derivations $D\longmapsto D^p$. In terms of the basis $X^{(\alpha)}\partial_i$, the restricted structure can be described by
\begin{equation*}
(X^{(\alpha)}\partial_i)^{[p]}=
\begin{cases}
X_i\partial_i&\text{if $\alpha=\epsilon_i$},\\
0&\text{otherwise}.
\end{cases}
\end{equation*}
Under this restricted structure, the graded Lie algebra $W(m;\unit)$ gives a $\Z$-graded restricted Lie algebra in the sense of Definition \ref{def:GRLA}. 
\end{rem}

For any $1\le i\le m$, we define the map $v_i:A(m;\bfn)\lto\Z\cap [0,p^{n_i}]$ to be
\begin{equation*}
v_i(f)\Def
\begin{cases}
\min\{\alpha_i\,|\,c_{\alpha}\neq 0\}&\text{if $f\neq 0$},\\
p^{n_i}&\text{if $f=0$}
\end{cases}
\end{equation*}
for any $f=\sum_{0\le\alpha\le\bfp^{\bfn}-\unit}c_{\alpha}X^{(\alpha)}\in A(m;\bfn)$. Moreover, 
let us define the map
\begin{equation*}
d_i:W(m;\bfn)\lto\Z\cap [-1,p^{n_i}]
\end{equation*}
to be
\begin{equation}\label{eq:d_i}
d_i(D)\Def 
\begin{cases}
\min\{v_i(f_j)-\delta_{ij}\,|\,j=1,\dots,m\}&\text{if $D\neq 0$},\\
p^{n_i}&\text{if $D=0$}
\end{cases}
\end{equation}
for any $D=\sum_{j=1}^mf_j\partial_j\in W(m;\bfn)$. 
Note that $d_i(D)=p^{n_i}$ if and only if $D=0$. 
For later use, we put some computations here.

\begin{lem}\label{lem:W(m;n) formula}
Suppose that $k$ is a field of characteristic $p>0$.  Fix $1\le i\le m$. 
\begin{enumerate}
\renewcommand{\labelenumi}{(\arabic{enumi})}
\item Any non-zero element $0\neq D\in W(m;\bfn)$ can be uniquely written as 
\begin{equation*}
D=\sum_{\beta_0\le \beta\le p^{n_i}-1}X_i^{(\beta)}D_{\beta}
\end{equation*}
so that $D_{\beta_0}\neq 0$ and $[\partial_i,D_{\beta}]=0$ for any $\beta_0\le\beta\le p^{n_i}-1$. Then we have
\begin{equation*}
d_i(D)=
\begin{cases}
\beta_0&\text{if $D_{\beta_0}(X_i)=0$},\\
\beta_0-1&\text{if $D_{\beta_0}(X_i)\neq 0$}.
\end{cases}
\end{equation*}

\item For any $D,E\in W(m;\bfn)$, we have
\begin{equation*}
d_i(D+E)\ge \min\{d_i(D),d_i(E)\}.
\end{equation*}

\item For any $D,E\in W(m;\bfn)$, we have
\begin{equation*}
d_i([D,E])\ge\min\{p^{n_i},d_i(D)+d_i(E)\}.
\end{equation*}
Hence, $[D,E]=0$ if $d_i(D)+d_i(E)\ge p^{n_i}$. 
\item For any $D,E\in W(m;\bfn)$ satisfying $d_i(D)+d_i(E)-1\ge p^{n_i}$, we have 
\begin{equation*}
\ad(D)\ad(E)=0.
\end{equation*}
\end{enumerate}
\end{lem}

\begin{proof}
(1) The first assertion is immediate from the definition of the Witt algebra $W(m;\bfn)$. Then the last assertion follows from the definition of the map $d_i$~(cf.\ (\ref{eq:d_i})). 

(2) If $D=0$ or $E=0$, then this is true. So let us assume that $D\neq 0$ and $E\neq 0$. Let us write
\begin{equation*}
D=\sum_{\beta_0\le \beta\le p^{n_i}-1}X_i^{(\beta)}D_{\beta},\quad E=\sum_{\gamma_0\le \gamma\le p^{n_i}-1}X_i^{(\gamma)}E_{\gamma}
\end{equation*} 
so that $D_{\beta_0}\neq 0$ and $E_{\gamma_0}\neq 0$ and $[\partial_i,D_{\beta}]=[\partial_i,E_{\gamma}]=0$. Without loss of generality, we may assume that $\beta_0\le \gamma_0$. 
If $\beta_0<\gamma_0$, by (1), we have $d_i(D+E)=d_i(D)=\min\{d_i(D),d_i(E)\}$, Hence, the inequality holds. On the other hand, let us assume that $\beta_0=\gamma_0$. Then we have
\begin{equation*}
D+E=\sum_{\beta_0\le \beta\le p^{n_i}-1}X_i^{(\beta)}(D_{\beta}+E_{\beta}).
\end{equation*}
If $D_{\beta_0}+E_{\beta_0}=0$, then we have $d_i(D+E)\ge \beta_0\ge d_i(D),d_i(E)$. Hence, the inequality holds. Let us assume that $D_{\beta_0}+E_{\beta_0}\neq 0$. If $(D_{\beta_0}+E_{\beta_0})(X_i)=0$, then by (1), we have $d_i(D+E)=\beta_0\ge d_i(D),d_i(E)$, hence the inequality. If $(D_{\beta_0}+E_{\beta_0})(X_i)\neq 0$, then $D_{\beta_0}(X_i)\neq 0$ or $E_{\beta_0}(X_i)\neq 0$, hence, by (1), we have $d_i(D+E)=\beta_0-1=\min\{d_i(D),d_i(E)\}$. This completes the proof.

(3) If $[D,E]=0$, the inequality holds. Hence, let us suppose that $[D,E]\neq 0$. 
Let us write as in (1), 
\begin{equation*}
D=\sum_{\beta_0\le \beta\le p^{n_i}-1}X_i^{(\beta)}D_{\beta},\quad E=\sum_{\gamma_0\le \gamma\le p^{n_i}-1}X_i^{(\gamma)}E_{\gamma}
\end{equation*} 
so that $D_{\beta_0}\neq 0$ and $E_{\gamma_0}\neq 0$ and $[\partial_i,D_{\beta}]=[\partial_i,E_{\gamma}]=0$. Then we have
\begin{equation*}
\lbrack D,E\rbrack
=\sum_{\beta_0\le\beta\le p^{n_i}-1}\sum_{\gamma_0\le\gamma\le p^{n_i}-1}[X_i^{(\beta)}D_{\beta},X_i^{(\gamma)}E_{\gamma}].
\end{equation*}
First assume that $[X_i^{(\beta_0)}D_{\beta_0},X_i^{(\gamma_0)}E_{\gamma_0}]=0$. Note that
\begin{equation*}
\begin{aligned}
&~[X_i^{(\beta_0+1)}D_{\beta_0+1},X_i^{(\gamma_0)}E_{\gamma_0}]+[X_i^{(\beta_0)}D_{\beta_0},X_i^{(\gamma_0+1)}E_{\gamma_0+1}]\\
=&~X_i^{(\beta_0+1)}X_i^{(\gamma_0-1)}D_{\beta_0+1}(X_i)E_{\gamma_0}-X_i^{(\beta_0)}X_i^{(\gamma_0)}E_{\gamma_0}(X_i)D_{\beta_0+1}\\
&+X_i^{(\beta_0)}X_i^{(\gamma_0)}D_{\beta_0}(X_i)E_{\gamma_0+1}-X_i^{(\beta_0-1)}X_i^{(\gamma_0+1)}E_{\gamma_0+1}(X_i)D_{\beta_0}\\
=&~X_i^{(\beta_0+\gamma_0)}F_{\beta_0+\gamma_0}
\end{aligned}
\end{equation*} 
with
\begin{equation*}
\begin{aligned}
F_{\beta_0+\gamma_0}
=&~\binom{\beta_0+\gamma_0}{\beta_0+1}D_{\beta_0+1}(X_i)E_{\gamma_0}-\binom{\beta_0+\gamma_0}{\gamma_0}E_{\gamma_0}(X_i)D_{\beta_0+1}\\
&+\binom{\beta_0+\gamma_0}{\beta_0}D_{\beta_0}(X_i)E_{\gamma_0+1}-\binom{\beta_0+\gamma_0}{\beta_0-1}E_{\gamma_0+1}(X_i)D_{\beta_0}.
\end{aligned}
\end{equation*}
If $E_{\gamma_0}(X_i)=D_{\beta_0}(X_i)=0$, then $F_{\beta_0+\gamma_0}(X_i)=0$. Therefore, the result (1) implies that
\begin{equation*}
d_i([D,E])\ge\min\{p^{n_i},\beta_0+\gamma_0\}=\min\{p^{n_i},d_i(D)+d_i(E)\}.
\end{equation*}
If $F_{\beta_0+\gamma_0}(X_i)\neq 0$, then $E_{\gamma_0}(X_i)\neq 0$ or $D_{\beta_0}(X_i)\neq 0$. Thus, again by (1), we have 
\begin{equation*}
d_i([D,E])\ge\min\{p^{n_i},\beta_0+\gamma_0-1\}\ge\min\{p^{n_i},d_i(D)+d_i(E)\}.
\end{equation*}
It remains to handle the case when $[X_i^{(\beta_0)}D_{\beta_0},X_i^{(\gamma_0)}E_{\gamma_0}]\neq 0$. Note that
\begin{equation*}
[X_i^{(\beta_0)}D_{\beta_0},X_i^{(\gamma_0)}E_{\gamma_0}]
=X_i^{(\beta_0+\gamma_0-1)}F_{\beta_0+\gamma_0-1}
\end{equation*}
with
\begin{equation*}
F_{\beta_0+\gamma_0-1}=\binom{\beta_0+\gamma_0-1}{\gamma_0}D_{\beta_0}(X_i)E_{\gamma_0}-\binom{\beta_0+\gamma_0-1}{\beta_0}E_{\gamma_0}(X_i)D_{\beta_0}.
\end{equation*}
As $F_{\beta_0+\gamma_0-1}\neq 0$, $D_{\beta_0}(X_i)\neq 0$ or $E_{\gamma_0}(X_i)\neq 0$. Without loss of generality, we may assume that $D_{\beta_0}(X_i)\neq 0$. 
If $F_{\beta_0+\gamma_0-1}(X_i)=0$, then by (1), we have
\begin{equation*}
d_i([D,E])\ge\min\{p^{n_i},\beta_0+\gamma_0-1\}\ge\min\{p^{n_i},d_i(D)+d_i(E)\}. 
\end{equation*}
If $F_{\beta_0+\gamma_0-1}(X_i)\neq 0$, we must have $D_{\beta_0}(X_i)E_{\gamma_0}(X_i)\neq 0$. Thus, again by (1), we have
\begin{equation*}
d_i([D,E])\ge\min\{p^{n_i},\beta_0+\gamma_0-2\}\ge\min\{p^{n_i},d_i(D)+d_i(E)\}. 
\end{equation*}
This completes the proof.

(4) Indeed, for any $F\in W(m;\bfn)$, by (3), 
\begin{equation*}
d_i(\ad(D)\ad(E)(F))\ge\min\{p^{n_i},d_i(D)+d_i(\ad(E)(F))\}\ge\min\{p^{n_i},d_i(D)+d_i(E)+d_i(F)\}.
\end{equation*}
As $d_i(F)\ge -1$, the condition $d_i(D)+d_i(E)-1\ge p^{n_i}$ thus implies $\ad(D)\ad(E)(F)=0$. This completes the proof. 
\end{proof}

\subsection{The graded Cartan type Lie algebras}\label{subsec:gr cartan}

Let $L$ be a Lie algebra over $k$. We set $L^{(0)}\Def L$.  Inductively, for any $i>0$, we define $L^{(i)}$ to be the derived subalgebra of $L^{(i-1)}$, \emph{i.e.}\ $L^{(i)}\Def [L^{(i-1)},L^{(i-1)}]$. Moreover, we set $L^{(\infty)}\Def\cap_{i=1}^{\infty}L^{(i)}$. As we have assumed that $L$ is finite dimensional, there exists an integer $n>0$ such that $L^{(\infty)}=L^{(n)}$. 

We continue to use the same notation as in \S\ref{subsec:witt}. We define
\begin{equation*}
\Omega^0(m;\bfn)\Def A(m;\bfn),\quad\Omega^1(m;\bfn)\Def\Hom_{A(m;\bfn)}(W(m;\bfn),A(m;\bfn)), 
\end{equation*}
and $\Omega^r(m;\bfn)\Def \wedge^r\Omega^1(m;\bfn)$ for $r>1$. 
Then we have the natural map $d:A(m;\bfn)\lto\Omega^1(m;\bfn)\,;\,f\longmapsto df$, where $df(D)\Def D(f)$ for any $D\in W(m;\bfn)$, and $dX_1,\dots,dX_m$ then form a free basis of $\Omega^1(m;\bfn)$.  Note that there exists a natural action of $W(m;\bfn)$ on $\Omega^1(m;\bfn)$, \emph{i.e.}\
\begin{equation*}
W(m;\bfn)\times\Omega^1(m;\bfn)\lto\Omega^1(m;\bfn)\,;\,(D,\omega)\longmapsto D(\omega(-))-\omega([D,-]), 
\end{equation*}
which canonically extends to an action on $\Omega^r(m;\bfn)$ for any
$r$.
This action is just the Lie derivative of differential geometry.

Let us consider the following differential forms (cf.\ \cite[\S4.2]{Strade})
\begin{equation}\label{eq:omega_S}
\omega_S=dX_1\wedge\dots\wedge dX_m,\quad m\ge 2.
\end{equation}
\begin{equation}\label{eq:omega_H}
\omega_H=\sum_{i=1}^rdX_i\wedge dX_{i+r},\quad m=2r.
\end{equation}
\begin{equation}\label{eq:omega_K}
\omega_K=dX_{2r+1}+\sum_{i=1}^{2r}\sigma(i)X_idX_{i'},\quad m=2r+1
\end{equation}
where
\begin{equation}\label{eq:i' and sigma(i)}
i'=
\begin{cases}
i+r\\
i-r
\end{cases}
~~\sigma(i)=
\begin{cases}
1&~~\text{if}~~1\le i\le r,\\
-1&~~\text{if}~~r+1\le i\le 2r.
\end{cases}
\end{equation}

\begin{rem}\label{rem:omega S m=2 0}
Note that in the case when $m=2=2r$, we have
$\omega_S=\omega_H$. Thus, in \cite{Strade} and also in other
references, the index $m$ is usually assumed to be $m\ge 3$ for
the differential form $\omega_S$. However, we basically do not
assume $m\neq 2$ in the sequel. This is convenient for handling non-graded simple Lie algebras~(cf.\ \S\ref{subsec:cartan simple}). 
In the case when $m=2$, we also have $S(2;\bfn)=H(2;\bfn)$ (see
below). 
\end{rem}

\subsubsection{The Special algebras}\label{subsubsec:S(m;n)}

Let $m\ge 2$ be an integer and $\bfn\in\Z_{>0}^m$. The
\textit{Special algebra} $S(m;\bfn)$ is the Lie subalgebra of
$W(m;\bfn)$ defined as
\begin{equation*}
S(m;\bfn)\Def\{D\in W(m;\bfn)\,;\,D(\omega_S)=0\}\end{equation*}
(cf.\ (\ref{eq:omega_S})). If $m\ge 3$, then the derived subalgebra $S(m;\bfn)^{(1)}$ is simple, and hence $S(m;\bfn)^{(\infty)}=S(m;\bfn)^{(1)}\subset S(m;\bfn)$~(cf.\ \cite[\S4.2, p.187]{Strade}). 
The Lie algebra $S(m;\bfn)^{(k)}~(k=0,1)$ has a natural grading induced by the one of $W(m;\bfn)$, \emph{i.e.} 
\begin{equation}\label{eq:S grade}
S(m;\bfn)^{(k)}=\bigoplus_{i\in\Z}S(m;\bfn)^{(k)}_i\quad\text{where}\quad S(m;\bfn)^{(k)}_i\Def S(m;\bfn)^{(k)}\cap W(m;\bfn)_i.
\end{equation}

The first derived subalgebra $S(m;\bfn)^{(1)}$ has a more explicit description. For any $1\le i,j\le m$, consider the map
\begin{equation*}
D_{i,j}:A(m;\bfn)\lto W(m;\bfn)~;~f\longmapsto \partial_{j}(f)\partial_i-\partial_i(f)\partial_j.
\end{equation*}
Then $S(m;\bfn)^{(1)}$ is the Lie subalgebra of $W(m;\bfn)$ spanned by $D_{i,j}(f)~(1\le i<j\le m,~f\in A(m;\bfn))$. This description immediately implies that
\begin{equation*}
\begin{gathered}
h(S(m;\bfn)^{(1)})=|\bfp^{\bfn}-\unit|-2\quad\text{and}\quad S(m;\bfn)^{(1)}_{|\bfp^{\bfn}-\unit|-2}=\sum_{i<j}k\cdot D_{i,j}(X^{(\bfp^{\bfn}-\unit)})\\
s(S(m;\bfn)^{(k)})=1~(k=0,1)\quad\text{and}\quad S(m;\bfn)^{(1)}_{-1}=S(m;\bfn)_{-1}=W(m;\bfn)_{-1}.
\end{gathered}
\end{equation*}
Moreover, by making use of the formula (\ref{eq:witt bracket}), one can easily see that
\begin{itemize}
\item For any $\alpha\in\N^m$ and any $1\le i,j,k\le n$, \begin{equation}\label{eq:S bracket}
[\partial_i,D_{j,k}(X^{(\alpha)})]=D_{j,k}(X^{(\alpha-\epsilon_i)}).
\end{equation}
\item For any $0\le i\le h(S(m;\bfn)^{(1)})$, we have
\begin{equation}\label{eq:S bracket 2}
[S(m;\bfn)^{(1)}_{-1},S(m;\bfn)^{(1)}_i]=S(m;\bfn)^{(1)}_{i-1}.
\end{equation}
\end{itemize}

\begin{lem}\label{lem: S(m,n)} If $k$ is of characteristic $p\ge 3$, and if we set
    $U\coloneqq S(m;\bfn)^{(1)}_{-1}$, then  $S(m;\bfn)^{(1)}$
    with its grading satisfies the conditions {\rm(G\, I), (G\,
    II), (G\, III)} in Corollary \ref{cor:AC GLA}.
\end{lem}
\begin{proof} Let us check the conditions one by one as follows:
    \begin{itemize} 
        \item (G\, I) follows from the fact that
            $s(S(m;\bfn)^{(1)})=1$;
        \item (G\, II) follows from the formula:
              $[S(m;\bfn)^{(1)}_{-1},S(m;\bfn)^{(1)}_i]=S(m;\bfn)^{(1)}_{i-1};$
        \item (G\, III) follows from the fact that
            $h(S(m;\bfn)^{(1)})=|\bfp^{\bfn}-\unit|-2\geq 2$ when
            $p\ge 3$ (note that $m\geq 2$ by definition).
    \end{itemize}
\end{proof}

\subsubsection{The Hamiltonian algebras}\label{subsubsec:H(2r;n)}

Here, we apply the previous notation as $m\Def 2r$. The
\textit{Hamiltonian algebra} $H(2r;\bfn)$ is the Lie subalgebra
of $W(2r;\bfn)$ defined as
\begin{equation*}
H(2r;\bfn)\Def\{D\in W(2r;\bfn)\,;\,D(\omega_H)=0\}\end{equation*}
(cf.\ (\ref{eq:omega_H})). Then the Lie algebra $H(2r;\bfn)^{(2)}$ is simple, and hence $H(2r;\bfn)^{(\infty)}=H(2r;\bfn)^{(2)}\subset H(2r;\bfn)$~(cf.\ \cite[\S4.2, p.188]{Strade}). 
The Lie algebra $H(2r;\bfn)^{(l)}~(0\le l\le 2)$ has a  natural grading induced by the one of $W(2r;\bfn)$ in the same way as (\ref{eq:S grade}), \emph{i.e.}\ by putting $H(2r;\bfn)^{(l)}_i\Def H(2r;\bfn)^{(l)}\cap W(m;\bfn)_{i}$ for $i\in\Z$.

Consider the linear map
\begin{equation*}
D_H:A(2r;\bfn)\lto W(2r;\bfn)~;~f\mapsto\sum_{i=1}^{2r}\sigma(i)\partial_i(f)\partial_{i'}
\end{equation*}
(cf.\ (\ref{eq:i' and sigma(i)})). Then the first derived subalgebra $H(2r;\bfn)^{(1)}$ is the Lie subalgebra of $W(2r;\bfn)$ generated by $D_H(X^{(\alpha)})$ for $0<\alpha\le\bfp^{\bfn}-\unit$. 
If we consider the Poisson bracket on $A({2r};\bfn)$,
\begin{equation*}
\{f,g\}\Def\sum_{i=1}^{2r}\sigma(i)\partial_i(f)\partial_{i'}(g), 
\end{equation*}
then we have
\begin{equation}\label{eq:poisson}
[D_H(f),D_H(g)]=D_H(\{f,g\}).
\end{equation}
The following are immediate for $l=1,2$.
\begin{equation*}
\begin{gathered}
h(H(2r;\bfn)^{(l)})=|\bfp^{\bfn}-\unit|-1-l\quad\text{and}\quad H(2r;\bfn)^{(l)}_{|\bfp^{\bfn}-\unit|-1-l}=
\begin{cases}
k\cdot D_H(X^{(\bfp^{\bfn}-\unit)})&\text{if $l=1$},\\
\sum_{i=1}^{2r}k\cdot D_{H}(X^{(\bfp^{\bfn}-\unit-\epsilon_i)})&\text{if $l=2$},
\end{cases}\\
s(H(2r;\bfn)^{(l)})=1\quad\text{and}\quad H(2r;\bfn)^{(l)}_{-1}=W(2r;\bfn)_{-1}.
\end{gathered}
\end{equation*}
Furthermore, by (\ref{eq:poisson}), one immediately obtain the following.
\begin{itemize}
\item For $1\le i\le 2r$ and $\alpha\in\N^{2r}$, 
\begin{equation}\label{eq:H bracket}
[D_H(X_i),D_H(X^{(\alpha)})]=D_H(\{X_i,X^{(\alpha)}\})=\sigma(i)D_H(X^{(\alpha-\epsilon_{i'})}).
\end{equation}
\item Therefore, for any $l=1,2$ and any $0\le i\le h(H(2r;\bfn)^{(l)})$, 
\begin{equation}\label{eq:H bracket 2}
[H(2r;\bfn)^{(l)}_{-1},H(2r;\bfn)^{(l)}_i]=H(2r;\bfn)^{(l)}_{i-1}.
\end{equation}
\end{itemize}

\begin{lem}\label{lem: H(m,n)} If $k$ is of characteristic $p>3$
    or $(p=3,r>1)$ or $(p=3,\bfn>\unit)$, and if we set
    $U\coloneqq H(2r;\bfn)^{(2)}_{-1}$, then  $H(2r;\bfn)^{(2)}$
    with its grading satisfies the conditions {\rm(G\, I), (G\,
    II), (G\, III)} in Corollary \ref{cor:AC GLA}.
\end{lem}
\begin{proof} Let us check the conditions one by one as follows:
    \begin{itemize} 
        \item (G\, I) follows from the fact that
            $s(H(2r;\bfn)^{(2)})=1$;
        \item (G\, II) follows from the formula:
              $[H(2r;\bfn)^{(2)}_{-1},H(2r;\bfn)^{(2)}_i]=H(2r;\bfn)^{(2)}_{i-1};$
        \item (G\, III) follows from the fact that
            $h(H(2r;\bfn)^{(2)})=|\bfp^{\bfn}-\unit|-3\geq 2$ when
            $p>3$
    or $(p=3,r>1)$ or $(p=3,\bfn>\unit)$.
    \end{itemize}
\end{proof}

\begin{rem}\label{rem:H alg}
It is known that if $p=3$, then the simple Hamiltonian algebra $H(2;\unit)^{(2)}$ is isomorphic to $\mathfrak{psl}_3$~(cf.\ \cite[\S4.4 (B)]{Strade}), hence it is of classical type~(cf.\ Remark \ref{rem:classical p=3}). 
\end{rem}


\subsubsection{The Contact algebras}

Finally we recall the description of the derived subalgebra $K(2r+1;\bfn)^{(1)}$ of the \textit{Contact algebra}
\begin{equation*}
K(2r+1;\bfn)\Def\{D\in W(2r+1;\bfn)\,;\,D(\omega_K)\in A(2r+1;\bfn)\omega_K\} 
\end{equation*}
(cf.\ (\ref{eq:omega_K})). For this aim, we follow
\cite[\S4.2]{Strade} and \cite[\S4.5]{sf88}. Here we always assume that $k$ is of characteristic $p>2$. 
Define the linear map
\begin{equation*}
D_K:A({2r+1};\bfn)\lto W({2r+1};\bfn)
\end{equation*}
to be
\begin{equation*}
D_K(f)=\sum_{j=1}^{2r}(\sigma(j)\partial_j(f)+X_{j'}\partial_{2r+1}(f))\partial_{j'}+(2f-\sum_{j=1}^{2r}X_j\partial_j(f))\partial_{2r+1}. 
\end{equation*}
We lift the Poisson bracket (\ref{eq:poisson}) on $A(2r;\bfn)$ to the one $\{-,-\}$ on $A({2r+1};\bfn)$ by composing with the projection $A({2r+1};\bfn)\lto A({2r};\bfn)$ given by $X_{2r+1}\mapsto 0$.

According to \cite[\S4.5, Theorem 5.5]{sf88}, we
have
\begin{itemize}
    \item $K({2r+1};\bfn)^{(1)}$ is a simple Lie algebra;
    \item $K({2r+1};\bfn)^{(1)}$ can be described as follows:
\begin{equation*}
K({2r+1};\bfn)^{(1)}\Def
\begin{cases}
D_K(A({2r+1};\bfn))=\sum_{0\le \alpha\le \bfp^{\bfn}-\unit}k\cdot D_K(X^{(\alpha)})&\text{if $2r+1\not\equiv-3~(\text{mod  $p$})$},\\
\sum_{0\le \alpha< \bfp^{\bfn}-\unit}k \cdot D_K(X^{(\alpha)})&\text{if $2r+1\equiv-3~(\text{mod $p$})$}.
\end{cases}
\end{equation*}
\end{itemize}
\todo[inline]{Reorganized here. Is it OK?}
In contrast to the Special or Hamiltonian algebras, the grading of the Lie algebra $K({2r+1};\bfn)^{(1)}$ does not come from the one of the Witt algebra $W(2r+1;\bfn)$. The grading
\begin{equation*}
K({2r+1};\bfn)^{(1)}
=\bigoplus_{i\in\Z}K({2r+1};\bfn)^{(1)}_i
\end{equation*}
is in fact given by~(cf.\ \cite[\S4.5, Proposition 5.3 (2)]{sf88})
\begin{equation*}
K({2r+1};\bfn)^{(1)}_i
\Def\sum_{||\alpha||=i}k\cdot D_K(X^{(\alpha)}). 
\end{equation*}
Then, it turns out that the depth of $K({2r+1};\bfn)^{(1)}$ is $s=2$ and the height is $h=\sum_{i=1}^{2r+1}p^{n_i}+p^{n_{2r+1}}-(2r+4)$ if $2r+1\not\equiv-3~(\text{mod $p$})$, and $h=\sum_{i=1}^{2r+1}p^{n_{i}}+p^{n_{2r+1}}-(2r+5)$ otherwise. Moreover,
\begin{equation*}
\begin{aligned}
K({2r+1};\bfn)^{(1)}_{-1}&=\sum_{j=1}^{2r}k\cdot D_K(X_i),\\
K({2r+1};\bfn)^{(1)}_{-2}&=k\cdot D_K(1)=k\cdot \partial_{2r+1}=[K(2r+1,\bfn)^{(1)}_{-1},K(2r+1;\bfn)_{-1}^{(1)}].
\end{aligned}
\end{equation*}
More generally, by using the equations \cite[\S4.2 (4.2.12), (4.2.13)]{Strade}, one can get 
\begin{equation*}
[K(2r+1;\bfn)^{(1)}_{-1},K(2r+1;\bfn)^{(1)}_i]=K(2r+1;\bfn)^{(1)}_{i-1}
\end{equation*}
for $-1\le i\le h(K(2r+1;\bfn)^{(1)})$. See also \cite[4.1.1 Proposition]{Bois}. 
As $[K({2r+1};\bfn)^{(1)}_{-1},K({2r+1};\bfn)^{(1)}_{-1}]\neq 0$, one cannot take $U=K(2r+1;\bfn)_{-1}^{(1)}$ in Corollary \ref{cor:AC GLA} for the Contact algebra $K({2r+1};\bfn)^{(1)}$~(cf.\ Example \ref{ex:counter-example in char p=3}). Thus, we change the generators in the following way.

\begin{lem}\label{lem:K(2r+1;n)}
Let $k$ be a field of characteristic $p>3$. Let $L=K(2r+1;\bfn)^{(1)}$ with $h=h(L)$. Define the subspaces $U^{+},U^-\subset L$ to be
\begin{equation*}
\begin{aligned}
U^-&\Def\sum_{i=1}^{r}k\cdot D_K(X_i),\\
U^+&\Def\sum_{j=r+1}^{2r}k\cdot D_K(X_{j}^{(p^{n_j}-1)})+L_{h}.
\end{aligned}
\end{equation*}
Then the subspaces $U^{\pm}$ satisfy the conditions (I), (II) and (III) in Theorem \ref{thm:AC FLA} with respect to the adjoint representation $\ad:L\lto\Der(L)\subset\fgl_L$. 
\end{lem}

\begin{proof}
We  consider another bracket $\langle-,-\rangle$ on $A({2r+1};\bfn)$ defined by
\begin{equation}\label{eq:poisson bracket contact}
\langle X^{(\alpha)},X^{(\beta)}\rangle \Def \{X^{(\alpha)},X^{(\beta)}\}+\Bigl(||\beta||\binom{\alpha+\beta-\epsilon_{2r+1}}{\beta}-||\alpha||\binom{\alpha+\beta-\epsilon_{2r+1}}{\alpha}\Bigl)X^{(\alpha+\beta-\epsilon_{2r+1})},
\end{equation}
where we set
\begin{equation*}
||\alpha||\Def\sum_{j=1}^{2r}\alpha_i+2\alpha_{2r+1}-2.
\end{equation*}
According to \cite[\S4.2 (4.2.10)]{Strade} or \cite[\S4.5,
Proposition 5.2]{sf88}, we have
\[
[D_K(X^{(\alpha)}),D_K(X^{(\beta)})]=D_K(\langle
X^{(\alpha)},X^{(\beta)}\rangle). 
\]
\todo[inline]{Moved this part here as requested by the referee. Could you
check?}

Let us check the condition (I). For any $1\le i\le r$, as $i'=r+i$, we have
\begin{equation*}
[D_K(X_i),D_K(X_{r+i}^{(\alpha)})]=D_K(\langle X_i,X_{i'}^{(\alpha)}\rangle)\overset{(\ref{eq:poisson bracket contact})}{=}D_K(\{X_i,X_{i'}^{(\alpha)}\})=D_K(X_{i'}^{(\alpha-1)})
\end{equation*}
for any $\alpha>0$. Therefore, $L_{-1}$ is contained in the Lie subalgebra generated by $U^{\pm}$. As $L_h\subset U^+$, it follows that $L$ is generated by $U^{\pm}$, hence the condition (I) is satisfied.  

For any $1\le i<j\le r$, we have $j\neq i'$~(cf.\ (\ref{eq:i' and sigma(i)})), hence
\begin{equation*}
[D_K(X_i),D_K(X_j)]=D_K(\langle X_i,X_j\rangle)\overset{(\ref{eq:poisson bracket contact})}{=}D_K(\{X_i,X_j\})=0.
\end{equation*}
Moreover, for any $1\le i\le r$, we have $D_K(X_i)=\partial_{i'}+X_i\partial_{2r+1}$, which immediately implies that
\begin{equation*}
\ad(D_K(X_i))^{p^k}=\ad(\partial_{i'})^{p^k}
\end{equation*}
for any $k>0$. Thus, the condition (II) is fulfilled by the equation (\ref{eq:partial p^n=0}). 

Let us check the last condition (III). However, as one can easily check that $h=h(L)>s(L)=2$, by the same argument as in the proof of Corollary \ref{cor:AC GLA}, we have $\ad(L_h)^2=0$. Thus, it suffices to show that $\ad(D_K(X_j^{(p^{n_j}-1)}))^2=0$ on $L$ for each $r+1\le j\le 2r$. By definition, for each $r+1\le j\le 2r$, we have  $d_j(D_K(X_j^{(p^{n_j}-1)}))\ge p^{n_j}-2$~(cf.\ Lemma \ref{lem:W(m;n) formula}(1)). Therefore, the desired vanishing $\ad(D_K(X_j^{(p^{n_j}-1)}))^2=0$ follows from Lemma \ref{lem:W(m;n) formula}(4), here we need the assumption that $p>3$. Hence, the condition (III) is satisfied. This completes the proof. 
\end{proof}

\begin{rem}
If $L$ is a simple Lie subalgebra of $W(m;\bfn)$, then for any elements $D,E\in L\subseteq W(m;\bfn)$, the condition that
\begin{equation*}
\ad(D)\ad(E)=0
\end{equation*}
as an element of $\Der(W(m;\bfn))$ implies that the same is true as an element of $\Der(L)$. In particular, to check the condition (III) in Theorem \ref{thm:AC FLA} for $L$ with respect to the adjoint representation $\ad:L\lto\Der(L)$, one can freely use Lemma \ref{lem:W(m;n) formula}. 
\end{rem}

\begin{ex}\label{ex:counter-example in char p=3}
In Lemma \ref{lem:AC FLA}, the assumption that $D^2=0$ is not removable. Here,
we offer a counterexample in characteristic $p=3$. Let $k$ be a field of characteristic $p=3$ and $L=K(3;\unit)^{(1)}$ be the Contact algebra. 
Note that
\begin{equation*}
\begin{gathered}
D_K(X_1)=\partial_{2}+X_1\partial_3, \\
D_K(X_2)=-\partial_1+X_2\partial_3,\\
[D_K(X_1),D_K(X_2)]=D_K(1)=2\partial_3.
\end{gathered}
\end{equation*}
As $D_K(X_i)^3=0$ on $A(3;\unit)$ for $i=1,2$, we get the exponential map~(cf.\ Theorem \ref{thm:f_X})
\begin{equation*}
\delta_i:\G_a\lto\GL_{A(3;\unit)}\,;\,t\longmapsto 1+t D_K(X_i)+t^2D_K(X_i)^2/2
\end{equation*}
whose derivation is the natural inclusion $d\delta_i:k=\Lie(\G_a)\lto\fgl_{A(3;\unit)}\,;\, a\longmapsto a D_K(X_i)$.
Let $\Gamma\subset \GL_{A(3;\unit)(1)}$ be the subgroup scheme of height one associated with the restricted Lie subalgebra  $L=K(3;\unit)^{(1)}\subset\Der(A(3;\unit))\subset\fgl_{A(3;\unit)}$. Then the map
\begin{equation*}
\xi:\G_{a,k}\times\Gamma\lto\GL_{A(3;\unit)(1)}\,;\,(t,u)\longmapsto\delta_1(t)^{-1}u\delta_1(t)
\end{equation*}
does not have image in $\Gamma$. Indeed, if it does, we must have
\begin{equation}\label{eq:ex counter-ex in char p=3}
\xi(\{1\}\times\delta_2(\alpha_3))\subseteq\xi(\G_{a,k}\times\Gamma)\subseteq\Gamma. 
\end{equation}
As
\begin{equation*}
\xi(1,\delta_2(v))=\delta_1(1)^{-1}\delta_2(v)\delta_1(1)=1+v\delta_1(1)^{-1}D_K(X_2)\delta_1(1)+v^2(\delta_1(1)^{-1}D_K(X_2)\delta_1(1))^2/2,
\end{equation*} 
the inclusion (\ref{eq:ex counter-ex in char p=3}) amounts to saying that
\begin{equation*}
E\Def \delta_1(1)^{-1}D_K(X_2)\delta_1(1)\in L.
\end{equation*}
However, 
\begin{equation*}
\begin{aligned}
E=&D_K(X_2)-[D_K(X_1),D_K(X_2)]-[D_K(X_1),[D_K(X_1),D_K(X_2)]]\\
~&-D_K(X_1)^2[D_K(X_2),D_K(X_1)]+[D_K(X_1),D_K(X_2)]D_K(X_1)^2+D_K(X_1)^2D_K(X_2)D_K(X_1)^2\\
=&D_K(X_2)-D_K(1)+2D_K(X_1)^2D_K(1)=D_K(X_2)+\partial_3+D_K(X_1)^2\partial_3,
\end{aligned}
\end{equation*}
hence $E\in L$ if and only if $\nabla\Def D_K(X_1)^2\partial_3=(\partial_2+X_1\partial_3)^2\partial_3=\partial_2^2\partial_3-X_1\partial_2\partial_3^2\in L$. The latter condition does not hold. Indeed, 
\begin{equation*}
\nabla(X_2X_3^{(2)})=-X_1\neq 0=X_2\nabla(X_3^{(2)})+X_3^{(2)}\nabla(X_2),
\end{equation*}
hence $\nabla\not\in\Der(A(3;\unit))$, which implies that $\nabla\not\in L$. Thus, the map $\xi$ does not factor through $\Gamma$.  
\end{ex}

\subsection{The non-graded Cartan type simple Lie algebras}\label{subsec:cartan simple}

\subsubsection{The non-graded Special algebras}\label{subsubsec:FLA type S}

We use the same notation as in \S\ref{subsubsec:S(m;n)}. Here, we assume $m\ge
2$~(cf.\ Remark \ref{rem:omega S m=2 0}). 
We introduce two classes of \textit{filtered Special algebras}.  
Namely, we consider the Lie algebras
\begin{equation*}
S(m;\bfn;\Phi(\tau))^{(1)},\quad S(m;\bfn;\Phi(l))~(1\le l\le m),
\end{equation*}
where each of them is a filtered Lie subalgebra of the Witt algebra $W(m;\bfn)$ described in the following way~(cf.~\cite[Theorems 6.3.7 and 6.3.10]{Strade}). 
\begin{equation}\label{eq:S(m;n;tau)}
S(m;\bfn;\Phi(\tau))^{(1)}=\sum_{i=1}^m k\cdot (1-X^{(\bfp^{\bfn}-\unit)})\partial_i+\sum_{i\ge 0}S(m;\bfn)^{(1)}_i.
\end{equation}
\begin{equation}\label{eq:S(m;n;l)}
\begin{aligned}
S(m;\bfn;\Phi(l))
=&\sum_{j\neq l}\sum_{\alpha\le \bfp^{\bfn}-\unit}k\cdot (D_{l,j}(X^{(\alpha)})-X_l^{(p^{n_l}-1)}X^{(\alpha)}\partial_j)\\
&+\sum_{j,k\neq l}\sum_{\alpha\le\bfp^{\bfn}-\unit}k\cdot  D_{j,k}(X^{(\alpha)})+\sum_{j\neq l}k\cdot \bigl(\prod_{i\neq j}X_i^{(p^{n_i}-1)}\bigl)\partial_j.
\end{aligned}
\end{equation}
In both the cases $L=S(m;\bfn;\Phi(\tau))^{(1)}$ and $L=S(m;\bfn;\Phi(l))$, we have 
\begin{equation*}
L_{(h)}=S(m;\bfn)^{(1)}_{(h)}=\sum_{i<j}k\cdot D_{ij}(X^{(\bfp^{\bfn}-\unit)})\quad\text{and}\quad h=h(L)=|\bfp^{\bfn}-\unit|-2.
\end{equation*}
Moreover, they are simple Lie algebras,  
and the associated graded Lie algebra $\gr L$ is naturally embedded as a $\Z$-graded Lie algebra into the graded Special algebra $S(m;\bfn)$ so that
\begin{equation*}
S(m;\bfn)^{(1)}\subset\gr L\subset S(m;\bfn).
\end{equation*}
We call both of them the \textit{Special algebras}.

\begin{lem}\label{lem:FLA type S}
Let $k$ be a field of characteristic $p>3$. 
\begin{enumerate}
\renewcommand{\labelenumi}{(\arabic{enumi})}
\item 
Let $L=S(m;\bfn;\Phi(\tau))^{(1)}$ with $h=h(L)$. Define the subspaces $U^-,U^+\subset L$ to be
\begin{equation*}
\begin{aligned}
U^-&=k\cdot (1-X^{(\bfp^{\bfn}-\unit)})\partial_1,\\
U^+&=\sum_{j=2}^mk\cdot X_1^{(p^{n_1}-2)}\partial_{j}+L_{(h)}.
\end{aligned}
\end{equation*}
Then $U^{\pm}$ satisfy the conditions (I), (II) and (III) in Theorem \ref{thm:AC FLA} with respect to the adjoint representation $\ad:L\lto\Der(L)\subset\fgl_L$.

\item Let $L=S(m;\bfn;\Phi(l))$ with $h=h(L)$. Define the subspaces $U^-,U^+\subset L$ to be
\begin{equation*}
\begin{aligned}
U^-=&\sum_{i\neq l}k\cdot \partial_i,\\
U^+=&\sum_{j\neq l}k\cdot (D_{l,j}\Bigl(\prod_{i\neq l}X_{i}^{(p^{n_{i}}-1)}\Bigl)-X^{(\bfp^{\bfn}-\unit)}\partial_{j})+L_{(h)}.
\end{aligned}
\end{equation*}
Then $U^{\pm}$ satisfy the conditions (I), (II) and (III) in Theorem \ref{thm:AC FLA} with respect to the adjoint representation $\ad:L\lto\Der(L)\subset\fgl_L$. 
\end{enumerate}
\end{lem}

\begin{proof}
(1) As $U^-$ is one dimensional, obviously we have $[U^-,U^-]=0$. 
Let $U\Def \sum_{j=2}^mk\cdot X_1^{(p^{n_1}-2)}\partial_{j}$. Then $U^+=U+L_{(h)}$. From the description (\ref{eq:S(m;n;tau)}), it follows that the associated graded Lie algebra $\gr L$ is isomorphic to $S(m;\bfn)^{(1)}$. Then $U^-$ and $U$ are lifts of the subspaces $V^-\Def k\cdot\partial_1\subset \gr_{-1} L=S(m;\bfn)^{(1)}_{-1}$ and $V\Def\sum_{j=2}^mk\cdot X_1^{(p^{n_1}-2)}\partial_{j}\subseteq \gr_{p^{n_1}-3}L=S(m;\bfn)^{(1)}_{p^{n_1}-3}$ respectively. By using (\ref{eq:witt bracket}), one can see $V^-$ and $V$ generate $S(m;\bfn)^{(1)}_{-1}=\gr_{-1}L$. Thus, by the equation (\ref{eq:S bracket 2}), it follows that  $V^{-}$ and $V+\gr_{h}L$ generate the graded Lie algebra $\gr L$. Since $V^-$ and $V+\gr_h L$ are the images of $U^-$ and $U^+=U+L_{(h)}$ respectively, by Lemma \ref{lem:GLA vs FLA}(1), we can conclude that $U^{\pm}$ generate $L$, hence the condition (I) is fulfilled. 

On the other hand, as $X^{(\bfp^{\bfn}-\unit)}\partial_1\in L_{(h)}\subset L_{(p^{n_1})}$, the equation (\ref{eq:partial p^n=0}) for $i=1$ implies that
\begin{equation*}
\ad((1-X^{(\bfp^{\bfn}-\unit)})\partial_1)^{p^{n_1}}(L_{(r)})\subseteq L_{(r+1)}
\end{equation*}
for all $r\ge -1$. Thus it follows that $\ad((1-X^{(\bfp^{\bfn}-\unit)})\partial_1)^{l}=0$ for sufficiently large $l>0$, hence the condition (II) is satisfied.

It remains to check the condition (III). Let $0\neq D\in U^+$ be a non-zero element. Then $d_i(D)\ge p^{n_1}-2$. Therefore, by Lemma \ref{lem:W(m;n) formula}(4), we have $\ad(D)^2=0$. This implies that $U^+$ satisfies the condition (III). This completes the proof.

(2) Fix $l'\neq l$. Let $L'$ be the Lie subalgebra of $L$ generated by $U^+$ and $U^-$. Note that, by (\ref{eq:witt bracket}) and (\ref{eq:S bracket}), we have
\begin{equation*}
\partial_l-X_l^{(p^{n_l}-1)}X_{l'}\partial_{l'}=\Bigl(\ad(\partial_{l'})^{p^{n_{l'}}-2}\cdot\prod_{i\neq l,l'}\ad(\partial_i)^{p^{n_i}-1}\Bigl)\Bigl(D_{l,l'}\Bigl(\prod_{j\neq l}X_{j}^{(p^{n_{j}}-1)}\Bigl)-X^{(\bfp^{\bfn}-\unit)}\partial_{l'}\Bigl)\in L'. 
\end{equation*}
Thus, it suffices to show that $U^-_0\Def U^-+k\cdot (\partial_l-X_l^{(p^{n_l}-1)}X_{l'}\partial_{l'})\subset L'$ and $U^+$ generates $L$. First let us show that
\begin{equation*}
D_{j,k}(X^{(\alpha)})\in L'
\end{equation*}
for any $j,k\neq l$ and any $\alpha<\bfp^{\bfn}-\unit$. As $j,k\neq l$, we have $D_{j,k}(X^{(\alpha)})=X_l^{(\alpha_l)}D_{j,k}(\prod_{i\neq l}X_i^{(\alpha_i)})$. Hence, 
\begin{equation*}
D_{j,k}(X^{(\alpha)})=
\begin{cases}
[\partial_l-X_l^{(p^{n_l}-1)}X_{l'}\partial_{l'},D_{j,k}(X^{(\alpha+\epsilon_l)})]&\text{if $\alpha_l<p^{n_l}-1$},\\
[\partial_i,D_{j,k}(X^{(\alpha+\epsilon_i)})] &\text{if $\alpha_i<p^{n_i}-1$ for $i\neq l$}.
\end{cases}
\end{equation*}
By induction on $|\alpha|$, this implies that for any $j,k\neq l$ and any $\alpha<\bfp^{\bfn}-\unit$, we have $D_{j,k}(X^{(\alpha)})\in L'$. 
Next let us show that
\begin{equation*}
\bigl(\prod_{i\neq j}X_i^{(p^{n_i}-1)}\bigl)\partial_j\in L'
\end{equation*}
for $j\neq l$. However, as $D_{l,j}\Bigl(\prod_{i\neq l}X_{i}^{(p^{n_{i}}-1)}\Bigl)-X^{(\bfp^{\bfn}-\unit)}\partial_{j}\in U^+$ and $\partial_j\in U^-$, we have
\begin{equation*}
\bigl(\prod_{i\neq j}X_i^{(p^{n_i}-1)}\bigl)\partial_j=-\ad(\partial_j)^{p^{n_j}-1}\Bigl(D_{l,j}\Bigl(\prod_{i\neq l}X_{i}^{(p^{n_{i}}-1)}\Bigl)-X^{(\bfp^{\bfn}-\unit)}\partial_{j}\Bigl)\in L'. 
\end{equation*}
Finally, let us show that
\begin{equation*}
D_{l,j}(X^{(\alpha)})-X_l^{(p^{n_l}-1)}X^{(\alpha)}\partial_{j}\in L'
\end{equation*}
for $j\neq l$ and for any $\alpha<\bfp^{\bfn}-\unit$. However, if $\alpha_l>0$, then $D_{l,j}(X^{(\alpha)})-X_l^{(p^{n_l}-1)}X^{(\alpha)}\partial_{j}
=D_{l,j}(X^{(\alpha)})$. Thus, we have
\begin{equation*}
D_{l,j}(X^{(\alpha)})-X_l^{(p^{n_l}-1)}X^{(\alpha)}\partial_{j}=\ad(\partial_{l}-X_{l}^{p^{n_l}-1}X_{l'}\partial_{l'})^{p^{n_l}-1-\alpha_l}\prod_{i\neq l}\ad(\partial_i)^{p^{n_i}-1-\alpha_i}(D_{l,j}(X^{(\bfp^{\bfn}-\unit)}))\in L'.
\end{equation*}
It remains to handle the case where $\alpha_l=0$. However, in this case, we have 
\begin{equation*}
D_{l,j}(X^{(\alpha)})-X_l^{(p^{n_l}-1)}X^{(\alpha)}\partial_{j}=\prod_{i\neq l}\ad(\partial_i)^{p^{n_i}-1-\alpha_i}\Bigl(D_{l,j}\Bigl(\prod_{i\neq l}X_{i}^{(p^{n_{i}}-1)}\Bigl)-X^{(\bfp^{\bfn}-\unit)}\partial_{j}\Bigl)\in L'
\end{equation*}
Therefore, the condition (I) is fulfilled. 

The second condition (II) follows from (\ref{eq:partial p^n=0}). Let $D$ be one of the following elements of $U^+$,
\begin{equation*}
(D_{l,j}\Bigl(\prod_{i\neq l}X_{i}^{(p^{n_{i}}-1)}\Bigl)-X^{(\bfp^{\bfn}-\unit)}\partial_j)~(j\neq l),~D_{i,j}(X^{(\bfp^{\bfn}-\unit)})~(1\le i<j\le m).
\end{equation*}
Then, as $m\ge 3$, $d_i(D)\ge p^{n_i}-2$ for some $1\le i\le m$. Thus, by Lemma \ref{lem:W(m;n) formula}(4), we get $\ad(D)^2=0$. This completes the proof. 
\end{proof}

\subsubsection{The non-graded Hamiltonian algebras}\label{subsubsec:FLA type H}

Next we handle two classes of filtered deformations of the graded Hamiltonian algebras $H(2r;\bfn)^{(i)}~(i\ge 0)$. 
The degree $2r=2$ case has been already discussed in \S\ref{subsubsec:FLA type S}~(cf.\ Remark  \ref{rem:omega S m=2 0}
). Thus, let us consider the case where $2r\ge 4$. 
Let $A((2r))$ be the completion of $A(2r)$~(cf.\ (\ref{eq:PD poly})) with respect to the filtration $\{A(2r)_{(j)}\}_{j=0}^{\infty}$, where $A(2r)_{(j)}\Def\sum_{|\alpha|\ge j}k\cdot X^{(\alpha)}$.
Let $W((2r))\Def\sum_{j=1}^{2r}A((2r))\partial_j$ and $\Omega^1((2r))\Def\Hom_{A((2r))}(W((2r)),A((2r)))$. Moreover, we set $\Omega^i((2r))\Def\wedge^i\Omega^1((2r))$ for $i>1$.

\begin{definition}(cf.\ \cite[Definition 6.4.1]{Strade}) A form $\omega=\sum_{i,j=1}^{2r}\omega_{ij}dX_i\wedge dX_j\in\Omega^2((2r))$ with $\omega_{ij}=-\omega_{ji}$ is called a \textit{Hamiltonian form} if $d\omega=0$ and $\Det(\omega_{ij})\in A((2r))^*$. 
\end{definition}

For a Hamiltonian form $\omega$, we set
\begin{equation*}
H((2r;\omega))\Def\{D\in W((2r))\,|\,D(\omega)=0\},
\end{equation*}
and, for $\bfn\in\Z_{\ge 1}^{2r}$,
\begin{equation}\label{eq:H(2r;n;omega)}
H(2r;\bfn;\omega)\Def H((2r;\omega))\cap W(2r;\bfn). 
\end{equation}
Note that if one chooses $\omega=\sum_{i=1}^{2r}\sigma(i)dX_i\wedge dX_{i'}=2\omega_H$~(cf.\ (\ref{eq:omega_H})), then $H(2r;\bfn;\omega)^{(\infty)}\simeq H(2r;\bfn)^{(2)}$~(cf.\ \S\ref{subsubsec:H(2r;n)}).

\begin{definition}(cf.\ \cite[\S6.5]{Strade})\label{def:Hamiltonian forms}~
\begin{enumerate}
\renewcommand{\labelenumi}{(\arabic{enumi})}
\item For a matrix $\alpha=(\alpha_{ij})\in M_{2r}(k)$ with $\alpha_{ij}=-\alpha_{ji}$ and $\bfn\in\Z^{2r}_{\ge 1}$, we define a Hamiltonian form $\omega(\alpha)=\omega(\alpha;\bfn)\in\Omega^2(2r;\bfn)$ to be
\begin{equation*}
\omega(\alpha)\Def\sum_{i=1}^{2r}\sigma(i)dX_i\wedge dX_{i'}+\sum_{i,j=1}^{2r}\alpha_{ij}X_i^{(p^{n_i}-1)}X_j^{(p^{n_j}-1)}dX_i\wedge dX_j.
\end{equation*}
Then the associated filtered simple Lie algebra $H(2r;\bfn;\omega(\alpha))^{(\infty)}$~(cf.\ (\ref{eq:H(2r;n;omega)})) is called a \textit{Hamiltonian algebra of first type}. 
\item For $1\le l\le 2r$ and $\bfn\in\Z_{\ge 1}^{2r}$, we define a Hamiltonian form $\omega_{H,l}=\omega_{H,l}(\bfn)\in\Omega^2((2r))\setminus\Omega^2(2r;\bfn)$ to be
\begin{equation*}
\omega_{H,l}\Def d(\exp(X_l^{(p^{n_l})})\sum_{j=1}^{2r}\sigma(j)X_jdX_{j'}).
\end{equation*}
Then the associated filtered simple Lie algebra $H(2r;\bfn;\omega_{H,l})^{(\infty)}$~(cf.\ (\ref{eq:H(2r;n;omega)})) is called a \textit{Hamiltonian algebra of second type}. 
\end{enumerate}
\end{definition}

Let us recall more concrete descriptions of them, which are given in \cite[Theorems 6.5.7 and 6.5.8]{Strade}. For a Hamiltonian form $\omega=\sum_{i,j=1}^{2r}\omega_{ij}dX_i\wedge dX_j\in\Omega^2((2r))$, let $(g_{\ij})\Def (\omega_{ij})^{-1}\in \GL_{2r}(A((2r)))$. Let us define the $k$-linear map 
\begin{equation*}
D_{H,\omega}:A((2r))\lto W((2r))
\end{equation*}
to be
\begin{equation}\label{eq:D H omega}
D_{H,\omega}(f)\Def-\sum_{i,j=1}^{2r}g_{ij}\partial_i(f)\partial_j. 
\end{equation}
Then for any $f,h\in A((2r))$, we have
\begin{equation}\label{eq:Lie bracket D H omega}
[D_{H,\omega}(f),D_{H,\omega}(h)]=D_{H,\omega}(D_{H,\omega}(f)(h))
\end{equation}
(cf.\ \cite[(6.5.5)]{Strade}). Note that $\Ker(D_{H,\omega})=k\cdot 1$, and we have $H((2r;\omega))=D_{H,\omega}(A((2r)))$~(cf.\ \cite[Theorem 6.5.6]{Strade}). In fact, for any $f,h\in A((2r))$, if one set
\begin{equation}\label{eq:poisson omega}
\{f,h\}_{\omega}\Def -D_{H,\omega}(f)(h)=\sum_{i,j=1}^{2r}g_{ij}\partial_i(f)\partial_j(h),
\end{equation}
then $\{~,~\}_{\omega}$ defines a Poisson bracket on $A((2r))$, for which $k\cdot 1=\Ker(D_{H,\omega})$ is the center of $A((2r))$, and the map $D_{H,\omega}$ induces an isomorphism of infinite dimensional Lie algebras
\begin{equation}\label{eq:A(2r) poisson}
(A((2r)), \{~,~\}_{\omega})/k\cdot 1\xrightarrow{~\simeq~}H((2r;\omega)).
\end{equation}
The description in terms of Poisson brackets holds similarly for the truncated Hamiltonian algebras $H(2r;\bfn;\omega)$.

Let us apply the description (\ref{eq:A(2r) poisson}) to the Hamiltonian algebra $H(2r;\bfn;\omega(\alpha))^{(\infty)}$ of first type~(cf.\ Definition \ref{def:Hamiltonian forms}(1)). In fact, according to \cite[Theorem 6.5.7]{Strade}, one can have
\begin{equation}\label{eq:H(2r;n;alpha)}
H(2r;\bfn;\omega(\alpha))^{(\infty)}=H(2r;\bfn;\omega(\alpha))^{(1)}=
\begin{cases}
\sum_{0<\beta\le\bfp^{\bfn}-\unit}k\cdot D_{H,\omega(\alpha)}(X^{(\beta)})&\text{if $\Det(\alpha)\neq 0$},\\
\sum_{0<\beta<\bfp^{\bfn}-\unit}k\cdot D_{H,\omega(\alpha)}(X^{(\beta)})&\text{if $\Det(\alpha)=0$}.
\end{cases}
\end{equation} 
Hence, the associated graded Lie algebra can be described by 
\begin{equation}\label{eq:gr H(2r;n;alpha)}
\gr (H(2r;\bfn;\omega(\alpha))^{(1)})\simeq
\begin{cases}
\sum_{0<\beta\le\bfp^{\bfn}-\unit}k\cdot D_{H}(X^{(\beta)})&\text{if $\Det(\alpha)\neq 0$},\\
\sum_{0<\beta<\bfp^{\bfn}-\unit}k\cdot D_{H}(X^{(\beta)})&\text{if $\Det(\alpha)=0$}.
\end{cases}
\end{equation} 
By (\ref{eq:D H omega}), the description of each $D_{H,\omega(\alpha)}(X^{(\beta)})$ is given by the inverse matrix $(g_{ij})$ of
\begin{equation*}
\omega(\alpha)=
\begin{pmatrix}
O&E_r\\
-E_r&O
\end{pmatrix}
+(\alpha_{ij}X_i^{(p^{n_i}-1)}X_j^{(p^{n_j}-1)})_{i,j=1}^{2r}. 
\end{equation*}
Thus, we have
\begin{equation}\label{eq:D H alpha}
D_{H,\omega(\alpha)}(X^{(\beta)})=D_H(X^{(\beta)})+\sum_{i,j=1}^{2r}X^{(\beta-\epsilon_i)}f_{ij}\partial_j,
\end{equation}
where each $f_{ij}$ is a linear combination of $X^{(\gamma)}$ with $\gamma_q=p^{n_q}-1$ for at least two indices $1\le q\le 2r$~(cf.\ \cite[(6.5.7)]{Strade}). Therefore,
\begin{equation}\label{eq:d_i D H alpha}
d_i(D_{H,\omega(\alpha)}(X^{(\beta)}))\ge \beta_i-2.
\end{equation} 
Furthermore, we have the following. 

\begin{lem}\label{lem:D H alpha p^n=0} 
For any $1\le i\le 2r$, we have 
$\ad(D_{H,\omega(\alpha)}(X_i))^{l}=0$ for $l\gg 0$.
\end{lem}

\begin{proof}
Thanks to the equation (\ref{eq:Lie bracket D H omega}), we have
\begin{equation*}
\ad(D_{H,\omega(\alpha)}(X_i))^{l}(D_{H,\omega(\alpha)}(g))=D_{H,\omega(\alpha)}(D_{H,\omega(\alpha)}(X_i)^{l}(g))
\end{equation*}
for $g\in A(2r;\bfn)$ and $l>0$. Therefore, it suffices to show that $D_{H,\omega(\alpha)}(X_i)^{l}=0$ on $A(2r;\bfn)$ for sufficiently large $l\gg 0$. By the equation (\ref{eq:D H alpha}), we have
\begin{equation*}
D_{H,\omega(\alpha)}(X_i)=\sigma(i)\partial_{i'}+\sum_{j=1}^{2r}f_{ij}\partial_{j},
\end{equation*}
where each $f_{ij}$ is a linear combination of $X^{(\gamma)}$ with $\gamma_q=p^{n_q}-1$ for at least two indices $1\le q\le 2r$. On the other hand, for any integer $k\ge 1$,
\begin{equation*}
\partial_{i'}^k\circ (f_{ij}\partial_j)=\sum_{m=0}^k\binom{k}{m}\partial_{i'}^m(f_{ij})\partial_{i'}^{k-m}\partial_j
\end{equation*}
in $\End(A(2r;\bfn))$, and each coefficient $\partial_{i'}^m(f_{ij})$ is a linear combination of $X^{(\gamma)}$ with $\gamma_q=p^{n_q}-1$ for some $q\neq i'$. As $\partial_{i'}^{p^{n_{i'}}}=0$ on $A(2r;\bfn)$, it follows that for any index $\beta\in\N^{2r}$, the element $D_{H,\omega(\alpha)}(X_i)^{p^{n_{i'}}}(X^{(\beta)})$ is a linear combination of $X^{(\beta')}$ with $\sum_{j\neq i'}\beta'_j>\sum_{j\neq i'}\beta_j$. This implies that $D_{H,\omega(\alpha)}(X_i)^{l}=0$ on $A(2r;\bfn)$ for $l>p^{n_{i'}}\sum_{j\neq i'}(p^{n_j}-1)$. This completes the proof.  
\end{proof}

\begin{lem}\label{lem:H(2r;n;alpha)}
Let $k$ be a field of characteristic $p>3$. Let $L\Def H(2r;\bfn;\omega(\alpha))^{(1)}$ with height $h\Def h(L)$. Fix $1\le l\le 2r$ such that $n_{l'}=\max\{n_{i}\,|\,1\le i\le 2r\}$. We define the subspaces $U^+,U^-\subset L$ to be
\begin{equation*}
\begin{aligned}
U^-=&k\cdot D_{H,\omega(\alpha)}(X_l),\\
U^+=&\sum_{j=1,j\neq l,l'}^{2r}k\cdot D_{H,\omega(\alpha)}\bigl(X_{l'}^{(p^{n_{l'}}-1)}X_j\bigl)+L_{(h)}.
\end{aligned}
\end{equation*}
If $p=5$, we assume that $\bfn\neq\unit$. Then the subspaces $U^{\pm}$ satisfy the conditions (I), (II) and (III) in Theorem \ref{thm:AC FLA} with respect to the adjoint representation $\ad:L\lto\Der(L)\subset\fgl_L$.
\end{lem}

\begin{proof}
Let $U\Def \sum_{j=1,j\neq l,l'}^{2r}k\cdot D_{H,\omega(\alpha)}\bigl(X_{l'}^{(p^{n_{l'}}-1)}X_j\bigl)$. Then $U^-$ and $U$ are lifts of the subspaces $V^-\Def k\cdot D_H(X_l)\subset \gr_{-1} L$ and $V\Def\sum_{j=1,j\neq l,l'}^{2r}k\cdot D_{H}\bigl(X_{l'}^{(p^{n_{l'}}-1)}X_j\bigl)\subseteq \gr_{p^{n_{l'}}-2}L$ respectively. By using (\ref{eq:H bracket}), one can see $V^-$ and $V$ generate $\sum_{j=1}^{2r}k\cdot D_H(X_j)=\gr_{-1}L$. Thus, by the equation (\ref{eq:H bracket 2}), it follows that $V^-$ and $V+\gr_h L$ generate $\gr L$. Therefore, by Lemma \ref{lem:GLA vs FLA}(1), we can conclude that $U^{-}$ and $U^+=U+L_{(h)}$ generate $L$, hence the condition (I) is fulfilled. The condition (II) follows from Lemma \ref{lem:D H alpha p^n=0}. 

It remains to check the condition (III). If one takes
\begin{equation*}
D=
\begin{cases}
D_{H,\omega(\alpha)}(X^{(\bfp^{\bfn}-\unit)})&\text{if $\Det(\alpha)\neq 0$},\\
D_{H,\omega(\alpha)}(X^{(\bfp^{\bfn}-\unit-\epsilon_i)})&\text{if $\Det(\alpha)=0$},
\end{cases}
\end{equation*}
then, by (\ref{eq:d_i D H alpha}), $d_j(D)\ge p^{n_j}-2$ for some $1\le j\le 2r$. Thus, Lemma \ref{lem:W(m;n) formula}(4) implies that $\ad(D)^2=0$. If $D=D_{H,\omega(\alpha)}\bigl(X_{l'}^{(p^{n_{l'}}-1)}X_j\bigl)\in U^+$, then $d_{l'}(D)\ge p^{n_{l'}}-3$. Therefore, again by Lemma \ref{lem:W(m;n) formula}(4), we have $\ad(D)^2=0$ unless $p^{n_{l'}}<6$. As $n_{l'}=\max\{n_i\,|\,1\le i\le 2r\,\}$, the inequality $p^{n_{l'}}<6$ only happens when $(p,\bfn)=(5,\unit)$. This completes the proof. 
\end{proof}

Let us consider the Hamiltonian algebra $H(2r;\bfn;\omega_{H,l})^{(\infty)}$ of second type~(cf.\ Definition \ref{def:Hamiltonian forms}(2)), in which case the corresponding Poisson bracket is a little bit harder to describe. For any $f\in A(2r;\bfn)$, we set
\begin{equation*}
D_{H,l}(f)\Def D_H(f)+\frac{\sigma(l)}{2}X_l^{(p^{n_l}-1)}(2f\partial_{l'}+\partial_{l'}(f)\sum_{j=1}^{2r}X_j\partial_j-\sum_{j=1}^{2r}X_j\partial_{j}(f)\partial_{l'})
\end{equation*}
(cf.\ \cite[(6.5.8)]{Strade}), where $D_H(f)=\sum_{i=1}^{2r}\sigma(i)\partial_{i}(f)\partial_{i'}$~(cf.\ \S\ref{subsubsec:H(2r;n)}). Then, according to \cite[Theorem 6.5.8]{Strade}, the Hamiltonian algebra $H(2r;\bfn;\omega_{H,l})^{(\infty)}$ can be described by
\begin{equation}\label{eq:H(2r;n;l)}
H(2r;\bfn;\omega_{H,l})^{(\infty)}=H(2r;\bfn;\omega_{H,l})^{(1)}=
\begin{cases}
\sum_{\alpha\le\bfp^{\bfn}-\unit}k\cdot D_{H,l}(X^{(\alpha)})&\text{if $r+1\not\equiv 0$ mod $p$},\\
\sum_{\alpha<\bfp^{\bfn}-\unit}k\cdot D_{H,l}(X^{(\alpha)})&\text{if $r+1\equiv 0$ mod $p$.}
\end{cases}
\end{equation} 
Note that for any $\alpha\in\N^{2r}$ and $1\le i\le 2r$, we have
\begin{equation}\label{eq:d_i D H l}
d_i(D_{H,l}(X^{(\alpha)}))\ge \alpha_i-1 
\end{equation}
(cf.\ Lemma \ref{lem:W(m;n) formula}(1)).

\begin{lem}\label{lem:H(2r;n;l)}
Let $k$ be a field of characteristic $p>3$. Let $L\Def H(2r;\bfn;\omega_{H,l})^{(1)}$ with height $h\Def h(L)$. We define the subspaces $U^+,U^-\subset L$ to be
\begin{equation*}
\begin{aligned}
U^-=&\sum_{i=1,i\neq l,l'}^rk\cdot D_{H,l}(X_i)+k\cdot D_{H,l}(X_l),\\
U^+=&\sum_{j=1,j\neq l,l'}^{r}k\cdot D_{H,l}\bigl(X_{j'}^{(p^{n_{j'}}-1)}\bigl)+k\cdot D_{H,l}(X_{l'}^{(p^{n_{l'}}-1)})\\
&+
\begin{cases}
k\cdot D_{H,l}(X^{(\bfp^{\bfn}-\unit)})&\text{if $r+1\not\equiv 0$ mod $p$},\\
\sum_{i=1}^{2r}k\cdot D_{H,l}(X^{(\bfp^{\bfn}-\unit-\epsilon_i)})&\text{if $r+1\equiv 0$ mod $p$.}
\end{cases}
\end{aligned}
\end{equation*}
Then $U^{\pm}$ satisfy the conditions (I), (II) and (III) in Theorem \ref{thm:AC FLA} with respect to the adjoint representation $\ad:L\lto\Der(L)\subset\fgl_L$. 
 \end{lem}

\begin{proof}
First note that
\begin{equation}\label{eq:lem H(2r;n;l)}
\begin{gathered}
D_{H,l}(X_i)=\sigma(i)\partial_{i'}+\frac{\sigma(l)}{2}X_l^{(p^{n_l}-1)}X_{i}\partial_{l'}, \quad i\neq l,l'\\
D_{H,l}(X_l)=\sigma(l)\partial_{l'},\quad D_{H,l}(X_{l'})=\sigma(l')\partial_{l}+\frac{\sigma(l)}{2}X_{l}^{(p^{n_l}-1)}(X_{l'}\partial_{l'}+\sum_{j=1}^{2r}X_j\partial_j). 
\end{gathered}
\end{equation}
In particular, this immediately implies that $[U^-,U^-]=0$. Let us prove that $U^+$ and $U^-$ generate $L$. Notice that
\begin{equation*}
D_{H,l}(1)=\ad(D_{H,l}(X_l))^{p^{n_{l'}}-1}(D_{H,l}(X_{l'}^{(p^{n_{l'}}-1)})).
\end{equation*}
Thus, by Lemma \ref{lem:GLA vs FLA}(1), it suffices to show that the images of $U^{\pm}$ and $D_{H,l}(1)$ in $\gr L$ generate the graded Lie algebra $\gr L$. By the description (\ref{eq:H(2r;n;l)}), we have
\begin{equation*}
\gr L\simeq
\begin{cases}
\sum_{0<\alpha\le\bfp^{\bfn}-\unit}k\cdot D_{H}(X^{(\alpha)})+k\cdot X_{l}^{(p^{n_l}-1)}\partial_{l'}&\text{if $r+1\not\equiv 0$ mod $p$},\\
\sum_{0<\alpha<\bfp^{\bfn}-\unit}k\cdot D_{H}(X^{(\alpha)})+k\cdot X_{l}^{(p^{n_l}-1)}\partial_{l'}&\text{if $r+1\equiv 0$ mod $p$},
\end{cases}
\end{equation*} 
and the images $\overline{U}^{\pm}$ in $\gr L$ are thus given by
\begin{equation*}
\begin{aligned}
\overline{U}^-=&\sum_{i=1,i\neq l,l'}^rk\cdot D_{H}(X_i)+k\cdot D_{H}(X_l),\\
\overline{U}^+=&\sum_{j=1,j\neq l,l'}^{r}k\cdot D_{H}\bigl(X_{j'}^{(p^{n_{j'}}-1)}\bigl)+k\cdot D_{H}(X_{l'}^{(p^{n_{l'}}-1)})\\
&+
\begin{cases}
k\cdot D_{H}(X^{(\bfp^{\bfn}-\unit)})&\text{if $r+1\not\equiv 0$ mod $p$},\\
\sum_{i=1}^{2r}k\cdot D_{H}(X^{(\bfp^{\bfn}-\unit-\epsilon_i)})&\text{if $r+1\equiv 0$ mod $p$,}
\end{cases}
\end{aligned}
\end{equation*}
and $\overline{D_{H,l}(1)}=X_{l}^{(p^{n_l}-1)}\partial_{l'}$ in $\gr L$. By the equations (\ref{eq:H bracket}) and (\ref{eq:H bracket 2}), $\overline{U}^{\pm}$ and $\overline{D_{H,l}(1)}$ generate the graded Lie algebra $\gr L$. Hence, the condition (I) is fulfilled.

Moreover, by the equation (\ref{eq:lem H(2r;n;l)}) together with Lemma \ref{lem:W(m;n) formula}(4), for any $i\in\{1,\dots,r,l\}\setminus\{l'\}$, we have
\begin{equation*}
\ad(D_{H,l}(X_i))^{p^k}=\sigma(i)\ad(\partial_{i'})^{p^k}
\end{equation*}
for any $k>0$. In particular, we have $\ad(D_{H,l}(X_i))^{p^{n_{i'}}}=0$ for $i\in\{1,\dots,r,l\}\setminus\{l'\}$. Thus, the condition (II) is also satisfied. 

Finally, let $D$ be one of the following elements of $U^+$,
\begin{equation*}
D_{H,l}\bigl(X_{j'}^{(p^{n_{j'}}-1)}\bigl),\quad 
\begin{cases}
D_{H,l}(X^{(\bfp^{\bfn}-\unit)})&\text{if $r+1\not\equiv 0$ mod $p$},\\
D_{H,l}(X^{(\bfp^{\bfn}-\unit-\epsilon_i)})&\text{if $r+1\equiv 0$ mod $p$.}
\end{cases}
\end{equation*}
Then, by (\ref{eq:d_i D H l}), $d_j(D)\ge p^{n_j}-2$ for some $1\le j\le 2r$. Thus, by Lemma \ref{lem:W(m;n) formula}(4), we have $\ad(D)^2=0$, hence the condition (III) is fulfilled.  This completes the proof. 
\end{proof}

\subsection{The Melikian algebras in characteristic $p=5$}\label{subsec:p=5}

In the case when $k$ is of characteristic $p=5$, there exists a further class of simple Lie algebras, which are called the \textit{Melikian algebras}. 
Here, let us recall the definition. For each pair $\bfn=(n_1,n_2)$   of positive integers. The \textit{Melikian algebra} $\calM(\bfn)=\calM(n_1,n_2)$ are defined as follows. Firstly, as a $k$-vector space, we have
\begin{equation*}
\calM(\bfn)=A(2;\bfn)\oplus W(2;\bfn)\oplus \widetilde{W(2;\bfn)}, 
\end{equation*}
where $\widetilde{W(2;\bfn)}$ is a copy of the Witt algebra $W(2;\bfn)$. The Melikian algebra $\calM(\bfn)$ has a Lie bracket so that the middle component $W(2;\bfn)$ becomes a Lie subalgebra of $\calM(\bfn)$ and it satisfies the following equations. Namely, 
\begin{equation*}
\begin{aligned}
\lbrack D,\widetilde{E}\rbrack&=\widetilde{[D,E]}+2{\rm div}(D)\widetilde{E},\\
[D,f]&= D(f)-2{\rm div}(D)f,\\
[f_1\widetilde{\partial}_1+f_2\widetilde{\partial}_2,g_1\widetilde{\partial}_1+g_2\widetilde{\partial}_2]&= f_1g_2-f_2g_1,\\
[f,\widetilde{E}]&= fE,\\
[f,g]&= 2(g\partial_2(f)-f\partial_2(g))\widetilde{\partial}_1+2(f\partial_1(g)-g\partial_1(f))\widetilde{\partial}_2, 
\end{aligned}
\end{equation*} 
for $D\in W(2;\bfn), \widetilde{E}\in\widetilde{W(2;\bfn)}$ and $f,g\in A(2;\bfn)$. Here, the map ${\rm div}:W(2;\bfn)\lto A(2;\bfn)$ is defined to be
\begin{equation*}
{\rm div}\Bigl(\sum_{i=1}^2f_i\partial_i\Bigl)\Def\sum_{i=1}^2\partial_i(f_i). 
\end{equation*}
Moreover, $\calM(\bfn)$ has the grading
\begin{equation}\label{eq:gr M}
\calM(\bfn)=\bigoplus_{i=-s}^{h}\calM(\bfn)_i
\end{equation}
for which $\calM(\bfn)$ becomes a $\Z$-graded Lie algebra with the depth $s(\calM(\bfn))=3$ and the height $h(\calM(\bfn))=3\cdot 5^{n_1+n_2}-7$. Moreover, we have
\begin{equation*}
\begin{aligned}
\calM(\bfn)_{-3}&=W(2;\bfn)_{-1}=k\cdot\partial_1+k\cdot\partial_2,\\
\calM(\bfn)_{-2}&=A(2;\bfn)_{0}=k\cdot 1,\\
\calM(\bfn)_{-1}&=\widetilde{W(2;\bfn)}_{-1}=k\cdot\widetilde{\partial}_1+k\cdot\widetilde{\partial}_2,\\
\calM(\bfn)_0&=\sum_{i,j=1,2}k\cdot X_i\partial_j,\\
\calM(\bfn)_{3\cdot 5^{n_1+n_2}-7}&=\widetilde{W(2;\bfn)}_{h(W^2)}=k\cdot X^{(\bfp^{\bfn}-\unit)}\widetilde{\partial}_1+k \cdot X^{(\bfp^{\bfn}-\unit)}\widetilde{\partial}_2. 
\end{aligned}
\end{equation*}

\begin{lem}\label{lem:melikian}
Let $k$ be a field of characteristic $p=5$. Then the Melikian algebra $\calM(\bfn)$ is generated by the subspaces $\calM(\bfn)_{-3}$ and $\calM(\bfn)_{h(\calM(\bfn))}$.
\end{lem}

\begin{proof}
By the definition of the bracket, we have
\begin{equation*}
\begin{aligned}
\lbrack\partial_j,X^{(\alpha)}\widetilde{\partial}_i\rbrack&=X^{(\alpha-\epsilon_j)}\widetilde{\partial}_i,\\
[X^{(\alpha)}\widetilde{\partial}_1,\widetilde{\partial}_2]&=X^{(\alpha)},\\
[1,\widetilde{E}]&=E.
\end{aligned}
\end{equation*}
This implies that $\calM(\bfn)$ can be generated by $\calM(\bfn)_{-3}$ and $\calM(\bfn)_{h(\calM(\bfn))}$.
\end{proof}

\begin{definition}\label{def:melikian}
A simple Lie algebra over an algebraically closed field $k$ of characteristic $p=5$ is of \textit{Melikian type} if it is isomorphic to the Melikian algebra $\calM(\bfn)$ for some pair $\bfn=(n_1,n_2)$ of positive integers. 
\end{definition}


\section{The conjecture for non-classical simple Lie algebras}\label{sec:application}

\subsection{Classification of simple Lie algebras}\label{subsec:KS conj}

Let $k$ be an algebraically closed field of characteristic $p>3$. 
Let us begin with the classification of \textit{Cartan type simple Lie algebras}. 
In the following, for any $X\in\{W,S,H,K\}$, we always consider $X(m;\bfn)^{(k)}(k\ge 0)$ as a $\Z$-graded Lie algebra with respect to the natural grading described in \S\ref{sec:Cartan}.

\begin{definition}(cf.\ \cite[Definition 4.2.4 and Theorem 4.2.7(2)]{Strade})\label{def:cartan type}
A filtered simple Lie algebra $L$ over $k$ is said to be of \textit{type} $X\in\{W,S,H,K\}$ if there exist $m\in\Z_{>0}$, $\bfn\in\Z^m_{>0}$ and an injective homomorphism $\psi:\gr L\hookrightarrow X(m;\bfn)$ of graded Lie algebras such that
\begin{equation*}
X(m;\bfn)^{(\infty)}\subseteq\psi(\gr L)\subseteq X(m;\bfn). 
\end{equation*}
A filtered simple Lie algebra over $k$ is said to be of \textit{Cartan type} if it is of type $X$ for some $X\in\{W,S,H,K\}$.   
A simple Lie algebra over $k$ is said to be of \textit{Cartan type} if it is isomorphic as a Lie algebra to a filtered simple Lie algebra of Cartan type. 
\end{definition}

\begin{thm}\label{thm:gr cartan type}
Let $k$ be an algebraically closed field of characteristic $p>3$. 
\begin{enumerate}
\renewcommand{\labelenumi}{(\arabic{enumi})}
\item {\rm (cf.\ \cite[Theorem 6.1.1(1)]{Strade})} The graded simple Cartan type Lie algebras 
\begin{equation*}
W(m;\bfn),~~S(m;\bfn)^{(1)}\ (m\ge 3),~~H(2r;\bfn)^{(2)},~~K(2r+1;\bfn)^{(1)} 
\end{equation*}
are not of classical type. 
\item {\rm (cf.\ \cite[Theorems 6.3.8, 6.5.1 and Corollaries 6.1.7, 6.6.2]{Strade})} Any $\Z$-graded simple Lie algebra over $k$ of Cartan type is isomorphic to one $X(m;\bfn)^{(\infty)}$ of the Lie algebras in (1). 
\item {\rm (cf.\ \cite[Corollaries 6.1.7 and 6.6.2]{Strade})} If $L$ is a filtered simple Lie algebra of type $X\in\{W,K\}$, then $L\simeq\gr L$ and hence it is isomorphic to $X(m;\bfn)^{(\infty)}$ for some $m\ge 1$ and $\bfn\in\Z^m_{>0}$.  
\end{enumerate}
\end{thm}

By Remarks \ref{rem:Witt alg} and \ref{rem:H alg}, the assumption that $p>3$ is necessary in Theorem \ref{thm:gr cartan type}(1). 
We call $m$ and $\bfn$ the \textit{degree} and  \textit{weight} respectively for $X(m;\bfn)^{(\infty)}$. Theorem \ref{thm:gr cartan type}(3) asserts that a non-graded Cartan type simple Lie algebra over $k$ is isomorphic to a  filtered simple Special or Hamiltonian algebra. In fact, a classification of such non-graded simple Cartan type Lie algebras is given in the following way. 

\begin{thm}{\rm (cf.\ \cite[Chapter 6]{Strade})}\label{thm:nongr cartan type}
Let $k$ be an algebraically closed field of characteristic $p>3$. 
\begin{enumerate}
\renewcommand{\labelenumi}{(\arabic{enumi})}
\item {\rm (cf.\ \cite[Theorems 6.3.8 and 6.3.10]{Strade})} Let $L$ be a filtered simple Lie algebra of type $S$ of degree $m\ge 2$. If $L$ is non-graded, then it is isomorphic to one of the filtered Special  algebras defined in \S\ref{subsubsec:FLA type S},
\begin{equation*}
S(m;\bfn;\Phi(\tau))^{(1)},\quad S(m;\bfn;\Phi(l))\,(1\le l\le m).
\end{equation*}

\item {\rm (cf.\ \cite[Corollary 6.4.10, Theorem 6.4.11 and Proposition 6.5.3]{Strade})}
Let $L$ be a filtered simple Lie algebra of type $H$ of degree $2r\ge 4$. If $L$ is non-graded, then it is isomorphic to one of the filtered Hamiltonian algebras defined in \S\ref{subsubsec:FLA type H}~(cf.\ Definition \ref{def:Hamiltonian forms}), 
\begin{equation*}
H(2r;\bfn;\omega(\alpha))^{(1)}~(\alpha\in M_{2r}(k) ~\text{with}~{}^t\alpha=-\alpha),\quad H(2r;\bfn;\omega_{H,l})^{(1)}~(1\le l\le 2r).
\end{equation*}
\end{enumerate}
\end{thm}

Now we can state the main ingredient for our proof of the main theorem. Our proof heavily relies on the classification theorem for simple Lie algebras in characteristic $p>3$, \emph{i.e.}\ the \textit{generalized Kostrikin--Shafarevich conjecture}, which can be stated as follows.

\begin{thm}{\rm (Block--Wilson--Strade--Premet, cf.\ \cite[Theorem 20.7.15]{Strade3})}\label{thm:KS conj}
Let $k$ be an algebraically closed field of characteristic $p>3$. Then any simple Lie algebra over $k$ is of classical, Cartan, or Melikian type.
\end{thm}

\subsection{The proof of the main theorem}\label{subsec:AC non classical}

Now let us prove the main result.

\begin{thm}\label{thm:AC Cartan}
Let $k$ be an algebraically closed field of characteristic $p>3$ and $L$  a
simple Lie algebra over $k$. Set $\Gamma\coloneqq
\mathfrak{G}(L_{[p]})$~\textup{(cf.\ \nameref{s:notation})}.  In the case where $p=5$, we assume that $L$ is not isomorphic to a simple Hamiltonian algebra $H(2r;\unit;\omega(\alpha))^{(\infty)}$ of first type with weight $\bfn=\unit$ and  degree $2r\ge 4$. Then $\Gamma\in \pi^\loc_A(\A^1_k)$. 
\end{thm}

\begin{proof}
Suppose given a simple Lie algebra $L$ over $k$. 
By Theorem \ref{thm:KS conj}, $L$ is of classical, Cartan or Melikian type. If $L$ is of classical type, then we have already seen the assertion~(cf.~Corollary \ref{cor:AC classical}). Let us assume that $L$ is graded of Cartan type~(cf.\ Definition \ref{def:cartan type}), \emph{i.e.}\ $L\simeq X(m;\bfn)^{(\infty)}$ for some $X\in\{W,S,H,K\}$, $m\ge 1$ and $\bfn\in\Z^m_{>0}$~(cf.\ Theorem \ref{thm:gr cartan type}(2)). If $X=K$, then the assertion follows from Theorem \ref{thm:AC FLA} together with Lemma \ref{lem:K(2r+1;n)}. 
For $X\in\{W,S,H\}$, we apply \ref{lem: W(m,n)}, \ref{lem:
S(m,n)}, \ref{lem: H(m,n)} to prove that the graded Lie algebra $L=X(m;\bfn)^{(\infty)}$ together with the subspace $U\Def L_{-1}=X(m;\bfn)^{(\infty)}_{-1}$ satisfies the conditions (G\,I), (G\,II) and (G\,III) in Corollary \ref{cor:AC GLA}.  
This completes the proof for all the graded Cartan type simple Lie algebras $L$. Let us assume $L$ is a non-graded Cartan type simple Lie algebra, in which case it is isomorphic to a filtered Special or Hamiltonian algebra~(cf.\ Theorem \ref{thm:gr cartan type}(4)). Therefore, thanks to Theorem \ref{thm:nongr cartan type}, the assertion is a consequence of Theorem \ref{thm:AC FLA} together with Lemmas \ref{lem:FLA type S}, \ref{lem:H(2r;n;alpha)} and \ref{lem:H(2r;n;l)}. This completes the proof for Cartan type simple Lie algebras.  

It remains to handle the Melikian algebra $L=\calM(\bfn)$ in characteristic $p=5$. It suffices to show that  the conditions (G\,I), (G\,II)  and (G\,III) in Corollary \ref{cor:AC GLA} are satisfied for the graded Lie algebra $\calM(\bfn)$ together with the subspace $U\Def\calM(\bfn)_{-3}\subset\calM(\bfn)_{\le -1}$. The condition (G\,I) is immediate from the vanishing $[\calM(\bfn)_{-3},\calM(\bfn)_{-3}]=\calM(\bfn)_{-6}=0$, where we use the depth $s(\calM(\bfn))=3$. 
The condition (G\,II) is due to Lemma \ref{lem:melikian}. Finally, as $p=5$, the last condition (III) is immediate from  
\begin{equation*}
h(\calM(\bfn))-s(\calM(\bfn))
=3\cdot 5^{n_1+n_2}-7-3>0,
\end{equation*}
This completes the proof.  
\end{proof}

For the graded simple Cartan type Lie algebras $L=X(m;\bfn)^{(\infty)}$ of type $X\neq K$, the proof of Theorem \ref{thm:AC Cartan} actually indicates the following slightly finer result.

\begin{cor}\label{cor:AC Cartan p=2,3}
Let $k$ be a perfect field of characteristic $p>0$. Let
$L=X(m;\bfn)^{(\infty)}$ be a graded Cartan type simple Lie algebra over $k$ of
type $X\neq K$. Set $G\coloneqq \mathfrak{G}(L_{[p]})$~\textup{(cf.\ \nameref{s:notation})}. 
\begin{enumerate}
\renewcommand{\labelenumi}{(\arabic{enumi})}
\item If $p> 2$ then $G\in \pi^\loc_A(\A^1_k)$. 
\item If $p=2$ then $G\in \pi^\loc_A(\A^1_k)$ unless
\begin{equation*}
(m;\bfn)=
\begin{cases}
(1,1),\, (1,2),\,(2,\bfn)\,(\bfn\le (2,2))&\text{if $X=W$},\\
(3,\unit)&\text{if $X=S$},\\
(2,\bfn)\,(\bfn\le (2,2)),\,(4,\unit)&\text{if $X=H$}.\end{cases}
\end{equation*}
\end{enumerate}
\end{cor}

\begin{proof}
Let $U\Def L_{-1}$ be taken as in the proof of Theorem \ref{thm:AC Cartan}. Then the conditions (G\,I) and (G\,II) in Corollary \ref{cor:AC GLA} are fulfilled with respect to  the subspace $U$. It suffices to check the condition (G\,III) in each cases. 

(1) It suffices to see the inequality $h(L)>s(L)$, which is valid unless $L=W(1;1)$ or $L=H(2;\unit)^{(2)}$ in characteristic $p=3$ case. If $p=3$ and $L=W(1;1)$ or $H(2;\unit)^{(2)}$, by Remarks \ref{rem:Witt alg} and \ref{rem:H alg}, the assertion is a consequence of Theorem \ref{thm:AC Sigma} for the Special linear groups $\Sigma=\SL_{n}$. This completes the proof. 

(2) Let $l\Def\max\{n_i\,|\,1\le i\le m\}$.
By (\ref{eq:partial p^n=0}), we have $\ad(U)^{p^l}=0$. It suffices to check the condition $h(L)>(2^{l-1}+1)s(L)$. One can easily check the inequality except in the listed cases.  
\end{proof}

In characteristic $p>5$ case, Theorem \ref{thm:AC Cartan} is unconditional, and hence we get the following consequence, which completes the proof of Theorem \ref{thm:main}.

\begin{cor}\label{cor:AC Cartan}
Let $k$ be an algebraically closed field of characteristic $p>5$. Then all finite local non-abelian simple $k$-group schemes belong to $\pi^\loc_A(\A^1_k)$.  
\end{cor}

\begin{proof}
This is immediate from Theorems \ref{thm:viviani} and \ref{thm:AC Cartan}. 
\end{proof}


{\small

\vspace{3mm}

\begin{flushleft}
Department of Mathematics, School of Engineering, Tokyo Denki University\\
5 Senju Asahi, Adachi, Tokyo 120-8551, Japan\\
\textit{E-mail adress}: \texttt{shusuke.otabe@mail.dendai.ac.jp}\\
\vspace{3mm}
Universit\`a degli Studi di Firenze, Dipartimento di Matematica e Informatica Ulisse Dini\\ 
Viale Morgagni 67/a, Firenze, 50134, Italy\\
\textit{E-mail adress}: \texttt{fabio.tonini@unifi.it}\\
\vspace{3mm}
The Chinese University of Hong Kong, Department of Mathematics\\ 
Shatin, New Territories, Hong Kong\\
\textit{E-mail adress}: \texttt{lzhang@math.cuhk.edu.hk}
\end{flushleft}

}

\end{document}